\documentclass[12pt]{amsart}
\newif\ifdraft
\draftfalse

\usepackage{color}
\ifdraft\usepackage[notref,notcite]{showkeys}\fi
\usepackage{graphicx,subfigure,suffix}

\usepackage{amssymb,amsmath,amsfonts,amssymb,mathtools}
\usepackage{latexsym}
\usepackage{cite}
-\marginparwidth \oddsidemargin -9mm \evensidemargin \oddsidemargin
\advance\hoffset by 5mm

\def \Rm {\mathbb R}
\def \Tm {\mathbb T}

\allowdisplaybreaks \numberwithin{equation}{section}

\newtheorem{theorem}{Theorem}[section]

\newtheorem{thm}{Theorem}[section]
\newtheorem{prop}[thm]{Proposition}
\newtheorem{lem}[thm]{Lemma}
\newtheorem{cor}[thm]{Corollary}

\newtheorem*{theorem*}{Theorem}
\newtheorem*{lemma*}{Lemma}
\newtheorem*{proposition*}{Proposition}
\newtheorem*{corollary*}{Corollary}

\theoremstyle{definition}
\newtheorem{definition}[thm]{Definition}

\theoremstyle{remark}
\newtheorem{remark}[theorem]{Remark}
\newtheorem*{remark*}{Remark}

\newcommand{\Chi}[1]{\Chi*{\{#1\}}}
\WithSuffix\newcommand\Chi*[1]{\chi_{\raise-.5ex\hbox{$\scriptstyle#1$}}}

\renewcommand{\leq}{\leqslant}
\renewcommand{\geq}{\geqslant}

\makeatletter
\@mparswitchfalse
\makeatother

\begin{document}
 \title{Suppression of chemotactic explosion by mixing}

\begin{abstract}
Chemotaxis plays a crucial role in a variety of processes in biology and ecology. In many instances, processes involving chemical attraction take place in fluids. One of the most studied PDE models of chemotaxis is given by the Keller-Segel equation, which describes a population density of bacteria or mold which attract chemically to substance they secrete. Solutions of the Keller-Segel equation can exhibit dramatic collapsing behavior, where density concentrates positive mass in a measure zero region. A natural question is whether presence of fluid flow can affect singularity formation by mixing the bacteria thus making concentration harder to achieve.

  In this paper, we consider the parabolic-elliptic Keller-Segel equation in two and three dimensions with additional advection term modeling ambient fluid flow. We prove that for any initial data, there exist incompressible fluid flows such that the solution to the equation stays globally regular. On the other hand, it is well known that when the fluid flow is absent, there exist initial data leading to finite time blow up. Thus presence of fluid flow can prevent the singularity formation.

  We discuss two classes of flows that have the explosion arresting property. Both classes are known as very efficient mixers. The first class are the relaxation enhancing (RE) flows of \cite{constantin2008diffusion}. These flows are stationary. The second class of flows are the Yao-Zlatos near-optimal mixing flows \cite{yao2014mixing}, which are time dependent. The proof is based on the nonlinear version of the relaxation enhancement construction of \cite{constantin2008diffusion}, and on some variations of global regularity estimate for the Keller-Segel model.
\end{abstract}
\author{Alexander Kiselev}
\address{\hskip-\parindent
Alexander Kiselev\\
Department of Mathematics\\
Rice University\\
Houston, TX 77005, USA}
\email{kiselev@rice.edu}
\author{Xiaoqian Xu}
\address{\hskip-\parindent
Xiaoqian Xu\\
Department of Mathematics\\
University of Wisconsin-Madison\\
Madison, WI 53706, USA
}
\email{xxu@math.wisc.edu}

\maketitle

\section{Introduction}

Chemotaxis is ubiquitous in biology and ecology. This term is used to describe motion where cells or species sense and attempt to move towards higher (or lower) concentration of some chemical. The first mathematically rigorous studies of chemotaxis effects have been by Patlak \cite{patlak1953random} and Keller-Segel \cite{keller1970initiation}, \cite{keller1971model}. The latter work involved derivation and first analysis of the Keller-Segel system, the most studied model of chemotaxis. The Keller-Segel equation describes a population of bacteria or mold that secrete a chemical and are attracted by it. In one version of the simplified parabolic-elliptic form, this equation can be written in $\mathbb{R}^d$ as (see e.g. \cite{perthame2006transport})
\begin{equation}\label{chemo}
\partial_t \rho - \Delta \rho + \nabla\cdot(\rho \nabla (-\Delta)^{-1}\rho)=0,\quad \rho(x,0)=\rho_0(x).
\end{equation}
The last term on the left hand side describes attraction of $\rho$ by the chemical whose concentration is given by $c(x,t)=(-\Delta)^{-1}\rho(x,t)$. The literature on the Keller-Segel equation is enormous. It is known that in dimensions larger than one, solutions to (\ref{chemo}) can concentrate finite mass in a measure zero region and so blow up in finite time. We refer to \cite{perthame2006transport}, \cite{horstmann20031970}, and \cite{horstmann200319701} for more details and further references.

Typically, chemotactic processes take place in fluid, and often the agents involved in chemotaxis are also advected by the ambient flow. Some of the examples involve monocytes using chemokine signalling to concentrate at a source of infection (see e.g. \cite{deshmane2009monocyte,taub1995monocyte}), sperm and eggs of marine animals practicing broadcast spawning in the ocean (see e.g. \cite{coll1994chemical,miller1985demonstration}), and other numerous instances in biology and ecology. Our goal in this paper is to study the possible effects resulting from interaction of chemotactic and fluid transport processes. Of particular interest to us is the possibility of suppression of finite time blow up due to the mixing effect of fluid flow.
The problem of chemotaxis in fluid flow has been studied before; for example, in a setting similar to ours \cite{KR1}, \cite{kiselev2012biomixing} studied the effect of chemotaxis and fluid advection on the efficiency of absorbing reaction. Moreover, in a series of papers \cite{lorz2010coupled}, \cite{duan2010global}, \cite{di2010chemotaxis}, \cite{liu2011coupled}, \cite{lorz2012coupled} a very interesting problem coupling chemotactic density with fluid
mechanics equations actively forced by this density has been considered in a variety of different settings. The active coupling makes the system more challenging to analyze, but in some cases intriguing results
involving global existence of weak solutions (the definition of which implies lack of the $\delta$-function blow up) have been proved.
These results, however, apply either in the setting where the initial data
is small (see e.g. \cite{lorz2012coupled}) or close to constant \cite{duan2010global}, or in the systems where both chemotactic equation and
the fluid equation have globally regular solutions if not coupled. In other words, to the best of our knowledge, there have been no rigorous results providing an example of suppression of the chemotactic explosion by fluid flow; only results showing that presence of fluid flow does not lead to blow up for the initial data that would not blow up without the flow.

In this paper, our main focus will be on the question whether incompressible fluid flow can arrest the finite time blow up phenomenon which is the key signature of the Keller-Segel model. There are two possible fluid flow effects that can be helpful in finite time blow up prevention. The first applies in infinite regions, where strong flow can help diffusion quickly spread the density so thin
that chemotactic effects become weak. The second effect is more universal and subtle to analyze, and involves mixing in a finite volume. In this case, the concentration may remain significant, but the flow is constantly mixing density and preventing chemotaxis from building up a concentration peak. We are primarily interested in the mixing effect, and so we will consider a finite region setting. It will also be convenient for us to adopt periodic boundary conditions and to consider the Keller-Segel equation with advection on a torus. This is not essential, and many of our results also apply on a finite region with Neumann, Dirichlet or Robin boundary conditions. 

Let us now state our main result. Since we are working on $\mathbb{T}^d$, we will define the concentration of the chemical by factoring out a constant background: $c(x,t)=(-\Delta)^{-1}(\rho(x,t)-\bar{\rho})$. Here $\rho(x,t)\in L^2$ is the species density, and $\bar{\rho}$ is its mean over $\mathbb{T}^d$. The inverse Laplacian can be defined on the Fourier side, or by an appropriate convolution as will be discussed below. Consider the equation
\begin{equation}\label{chemo1}
\partial_t\rho+(u\cdot\nabla) \rho -\Delta \rho + \nabla\cdot(\rho \nabla (-\Delta)^{-1}(\rho-\bar{\rho}))=0,\quad \rho(x,0)=\rho_0(x),\quad x\in \mathbb{T}^d.
\end{equation}
We will prove the following theorem.
\begin{thm}\label{main}
Given any initial data $\rho_0\geq 0$, $\rho_0\in C^{\infty}(\mathbb{T}^d)$, $d=2$ or $3$, there exist smooth incompressible flows $u$ such that the unique solution $\rho(x,t)$ of (\ref{chemo1}) is globally regular in time.
\end{thm}
\begin{remark}
1. We will provide more details on the specific choice of the flows later. \\
2. The restriction $d=2$ or $3$ is essential for our method to work. The case $d=4$ is in some sense borderline for our estimates: the needed estimates just fail. Extending our approach to this case, and especially to $d>4,$ would require new ideas. \\
Please check Remark~\ref{dim} after Theorem~\ref{L2cr} for more discussion. \\
3. As we already mentioned, the periodic boundary condition is not very essential. One can get a similar result in the bounded domain; one issue is that examples of stationary RE flows are not available in a general
bounded domain. On the other hand, the Yao-Zlatos flows can be constructed in general bounded domains and so can be used to achieve an analogous result in this more general setting.
See Remark \ref{noslip} for more details. \\
4. The theorem certainly holds under weaker assumptions on the regularity of the initial data. In this paper, for the sake of presentation, we do not make an effort to optimize the regularity conditions.
The scheme of our proof and the kinds of the flow examples that we have will make the connection between mixing properties of the flow and its ability to suppress the chemotactic blow up quite explicit.
\end{remark}

It is well known that the solutions to the Keller-Segel equation can form singularities in finite time. The first rigorous proof of this result in the case where domain is a two-dimensional disk was given by J{\"a}ger and Luckhaus \cite{jager1992explosions}. Their proof is based on radial geometry and comparison principles. Nagai \cite{nagai2001blowup} has provided a proof of finite time blow up in more general bounded domains. We have not found a finite time blow up proof for the periodic case in the literature. Although it can be obtained by modification of the existing arguments, in the appendix we will provide a short independent construction of such examples in the case of two spatial dimensions. This will imply that Theorem \ref{main} indeed provides examples of the suppression of chemotactic explosion by fluid mixing.

We note that fluid advection has been conjectured to regularize singular nonlinear dynamics before. The most notable example is the case of the $3$D Navier-Stokes and Euler equations. Constantin \cite{PC} has proved possibility of finite time blow up for
the $3$D Euler equation in $\Rm^3$ if the pure advection term in the vorticity formulation is removed from the equation. Hou and Lei have obtained numerical evidence for the finite time blow up in a system obtained from the $3$D Navier-Stokes equation by the removal of the pure transport terms \cite{lei2009stabilizing}. In fact, finite time blow up has been also proved rigorously in some related modified model settings \cite{hou2011singularity}, \cite{hou2012singularity}, \cite{hou2014finite}. Of course, the proof of the global regularity of the $3$D  Navier-Stokes remains an outstanding open problem, so whether the $3$D Navier-Stokes equation exhibits ``advection regularization'' is an open question. See also \cite{lartiti} for more discussion.
As another example of related philosophy, we mention the paper \cite{berestycki2010explosion} on the elliptic problem with ``explosion'' type reaction. There is no time variable and so no finite time blow up in this paper,
  but it shows that certain flows can significantly affect the ``explosion threshold'': namely, the value of the reaction coupling parameter beyond which there exist no regular positive solutions. To the best of our knowledge, the examples that we construct here are the first rigorous examples of the suppression of finite time blow up by fluid mixing in nonlinear evolution setting. It should be possible to extend our method to cover some other situations, which we will briefly discuss in the last section.

The paper is organized as follows.
In Section $2$ we prove the $L^2$ -based global regularity criterion that we will use. In particular, it guarantees global regularity as far as the $L^2$ norm of the solutions remains bounded. In Section $3$, we set up the strategy for controlling $L^2$ norm via $\dot{H}^1$ norm of the solution. Despite the presence of nonlinearity, due to the dissipation term, if the $\dot{H}^1$ norm is sufficiently large, the $L^2$ norm has to decay. Basically, what this implies is that the solution to the Keller-Segel equation cannot
blow up in the manner which makes $\dot{H}^1$ norm of the solution grow much faster than the $L^2$ norm.
In Section $4$, we prove the key result on approximation of the solution of the Keller-Segel equation by solution of pure advection equation on small time scales. This result sets up our strategy to use the fluid advection to drive up the $\dot{H}^1$ norm of the
solution to the full Keller-Segel equation by fluid mixing in order to keep the $L^2$ norm in check. Indeed, the highly mixing flows are known to increase the $\dot{H}^1$ norm of the solution of the advected passive scalar. If this solution stays close to the solution of the nonlinear problem long enough, this guarantees that the $\dot{H}^1$ norm of the nonlinear solution will be large, too.
In Section $5$, we put together all the components of the argument and prove that the relaxation enhancing flows of \cite{constantin2008diffusion} suppress chemotactic blow up.
We focus on the case of weakly mixing flows. In Section $6$ we outline another example of flows with analogous property, the Yao-Zlatos efficient mixing flows. Finally, in Section $7$ we briefly discuss possible future extensions. There are two appendices: the first one contains construction of finite time blow up in the Keller-Segel equation with periodic data, and the second includes sketches of proofs of some functional
inequalities that we use in the paper.

Throughout the paper, $C$ will stand for universal constants that may change from line to line.

\textbf{Acknowledgment.} AK has been partially supported by the NSF grant DMS-1412023. XX has been partially supported by the NSF grant DMS-1535653. We thank Eitan Tadmor, Changhui Tan, Yao Yao and Andrej Zlatos for helpful discussions.

\section{Global existence: the $L^2$ criterion}

In this section, we will show that to get the global regularity of (\ref{chemo1}), we only need to have certain control of the spatial $L^2$ norm of the solution. The following theorem is a direct analog of Theorem 3.1 in \cite{kiselev2012biomixing}, where it was proved in the $\mathbb{R}^2$ setting. We will provide a sketch of the proof for the sake of completeness. Throughout the paper, we will use notation $\dot{H}^s$ for the homogeneous Sobolev space in spatial coordinates, that is we set
$$
\|f\|_{\dot{H}^s}^2=\sum_{k\in\mathbb{Z}^d\setminus \{0\}}|k|^{2s}|\hat{f}(k)|^2.
$$

\begin{thm}\label{L2cr}
Suppose that $\rho_0 \in C^{\infty}(\mathbb{T}^d)$, $d=2$ or $3$, $\rho_0\geq 0$, and suppose that $u\in C^{\infty}$ is divergence free. Assume $[0,T]$ is the maximal interval of existence of the unique smooth solution $\rho(x,t)$ of the equation (\ref{chemo1}). Then we must have
\begin{equation}\label{13th}
\int_0^t\|\rho(\cdot, \tau)-\bar{\rho}\|_{L^2(\mathbb{T}^d)}^{\frac{4}{4-d}}d\tau\xrightarrow{t\rightarrow T}\infty.
\end{equation}
\end{thm}
\begin{remark}\label{dim}
In other words, the smooth solution can be continued as far as the integral in time of the appropriate power of the $L^2$ norm of solution in space stays finite. Note that the mean value of $\rho$ is conserved by the evolution, so $\bar{\rho}(\cdot,t)\equiv\bar{\rho}_0$. We will denote it $\bar{\rho}$ throughout the rest of the paper.
One may or may not include $\bar{\rho}$ into (\ref{13th}), these criteria are equivalent. One can verify that the scaling of (\ref{13th}) is sharp in the sense that it is a critical quantity for (\ref{chemo1}).
One way to see this criticality is through an informal scaling argument or dimensional analysis. Indeed, one can check that $\lambda^2\rho(\frac{x}{\lambda},\frac{t}{\lambda^2})$ is a solution to (\ref{chemo1}) (with a different period) for every $\lambda$ as long as $\rho(x,t)$ is a solution. The quantity in (\ref{13th}) with respect to this scaling. This observation also means that for $d\geq 4$ the boundedness of the $L^2$ norm may not be sufficient in order to get the regularity of the solution to (\ref{chemo1}).
\end{remark}

\begin{proof}
The existence and uniqueness of smooth local solution can be proved by standard methods, so we will focus on global regularity. Let $s\geq 2$ be integer. Multiply (\ref{chemo1}) by $(-\Delta)^s\rho$ and integrate. We get
\begin{align*}
\frac{1}{2}\partial_t \|\rho\|_{\dot{H}^s}^2&\leq \left|\int_{\mathbb{T}^d}(\nabla\rho)\cdot (\nabla(-\Delta)^{-1}(\rho-\bar{\rho}))(-\Delta )^s\rho \,dx \right|\\
&+\left|\int_{\mathbb{T}^d}\rho(\rho-\bar{\rho})(-\Delta)^s\rho \,dx\right|+C \|u\|_{C^s}\|\rho\|^2_{\dot{H}^s}-\|\rho\|^2_{\dot{H}^{s+1}}.
\end{align*}
Here we integrated by parts $s$ times and used incompressibility of $u$ to obtain
$$
\left|\int_{\mathbb{T}^d}(u\cdot \nabla)\rho(-\Delta)^s\rho \,dx\right|\leq C\|u\|_{C^s}\|\rho\|^2_{\dot{H}^s}.
$$
Consider the expression
$$
\int_{\mathbb{T}^d}\rho(\rho-\bar{\rho})(-\Delta)^s\rho dx.
$$
Integrating by parts, we can represent this integral as a sum of  terms of the form
$$
\int_{\mathbb{T}^d}D^l\rho D^{s-l}(\rho-\bar{\rho}) D^s \rho \,dx,
$$
where $l=0,1,...,s$ and $D$ denotes any partial derivative. By H\"{o}lder inequality, we have
$$
\int_{\mathbb{T}^d}D^l\rho D^{s-l}(\rho-\bar{\rho})D^s \rho \,dx\leq\|D^l \rho\|_{L^{p_l}}\|D^{s-l}(\rho-\bar{\rho})\|_{L^{q_l}}\|\rho\|_{\dot{H}^s},
$$
with any $2\leq p_l, q_l \leq \infty$ satisfying $p_l^{-1}+q_l^{-1}=1/2$. For any integer $0\leq m\leq n$, and mean zero $f\in C^{\infty}(\mathbb{T}^d)$, we have Gagliardo-Nirenberg inequality
\begin{equation}\label{14th}
\|D^m f\|_{L^p}\leq C\| f \|_{L^2}^{1-a}\|f\|_{\dot{H}^{n}}^a, \quad a=\frac{m-\frac{d}{p}+\frac{d}{2}}{n},
\end{equation}
which holds for $2\leq p \leq \infty$ unless $a=1,$ and if $a=1$ for $2\leq p<\infty$. We will sketch a short proof of (\ref{14th}) in the appendix to make the paper more self-contained.

Take $p_l=\frac{2s}{l}$, $q_l=\frac{2s}{s-l}$. Then for all $l>0$, applying (\ref{14th}), we get
$$
\|D^l\rho\|_{L^{p_l}}\|D^{s-l}(\rho-\bar{\rho})\|_{L^{q_l}}\leq C\|\rho-\bar{\rho}\|_{L^2}^{\frac{2s-d+4}{2(s+1)}}\|\rho\|_{\dot{H}^{s+1}}^{\frac{2s+d}{2(s+1)}}.
$$
In the $l=0$ case, we use
$$
\|\rho\|_{L^{\infty}}\leq \|\rho-\bar{\rho}\|_{L^{\infty}}+\bar{\rho}\leq C\|\rho-\bar{\rho}\|_{L^2}^{1-\frac{d}{2(s+1)}}\|\rho\|_{\dot{H}^{s+1}}^{\frac{d}{2(s+1)}}+\bar{\rho}.
$$
Therefore
$$
\left|\int_{\mathbb{T}^d}\rho(\rho-\bar{\rho})(-\Delta)^s\rho \,dx \right|\leq C\left(\|\rho-\bar{\rho}\|_{L^2}^{1-\frac{d}{2(s+1)}}\|\rho\|_{\dot{H}^{s+1}}^{\frac{d}{2(s+1)}}+\bar{\rho}\right) \|\rho\|^{\frac{s}{s+1}}_{\dot{H}^{s+1}}\|\rho-\bar{\rho}\|_{L^2}^{\frac{1}{s+1}}\|\rho\|_{\dot{H}^s}.
$$
Next, consider
$$
\int_{\mathbb{T}^d}(\nabla\rho)\cdot (\nabla(-\Delta)^{-1}(\rho-\bar{\rho}))(-\Delta)^s\rho \,dx.
$$
Integrating by parts $s$ times, we get terms that can be estimated similarly to the previous case, using the fact that the double Riesz transform $\partial_{ij}(-\Delta)^{-1}$ is bounded on $L^p$, $1<p<\infty$. The only exceptional terms that appear which have different structure are
$$
\int_{\mathbb{T}^d}(\partial_{i_1}...\partial_{i_s}\nabla\rho)\cdot(\nabla(-\Delta)^{-1}(\rho-\bar{\rho}))\partial_{i_1}...\partial_{i_s}\rho \,dx.
$$
But these can be reduced to
$$
\int_{\mathbb{T}^d}(\partial_{i_1}...\partial_{i_s}\rho)^2(\rho-\bar{\rho}) \,dx
$$
by another integration by parts, and estimated as before. Altogether, we get
\begin{equation}\label{15th}
\begin{split}
\frac{1}{2}\partial_t \|\rho\|_{\dot{H}^s}^2&\leq C\left(\|\rho-\bar{\rho}\|_{L^2}^{1-\frac{d}{2(s+1)}}\|\rho\|_{\dot{H}^{s+1}}^{\frac{d}{2(s+1)}}+\bar{\rho}\right) \|\rho\|^{\frac{s}{s+1}}_{\dot{H}^{s+1}}\|\rho-\bar{\rho}\|_{L^2}^{\frac{1}{s+1}}\|\rho\|_{\dot{H}^s}\\
&+C\|u\|_{C^s}\|\rho\|_{\dot{H}^s}^2-\|\rho\|_{\dot{H}^{s+1}}^2.
\end{split}
\end{equation}
Observe that
\begin{equation}\label{sobeasy56}
\|\rho\|_{\dot{H}^s}\leq \|\rho-\bar{\rho}\|_{L^2}^{\frac{1}{s+1}}\|\rho\|_{\dot{H}^{s+1}}^{\frac{s}{s+1}}.
\end{equation}
Split the first term on the right hand side of (\ref{15th}) into two parts, and estimate them as follows. First,
\begin{align*}
C\|\rho-\bar{\rho}\|_{L^2}^{\frac{2s-d+4}{2(s+1)}}\|\rho\|_{\dot{H}^{s+1}}^{\frac{2s+d}{2(s+1)}}\|\rho\|_{\dot{H}^s}&\leq\|\rho-\bar{\rho}\|_{L^2}\|\rho\|_{\dot{H}^{s+1}}^{\frac{d}{2}}\|\rho\|_{\dot{H}^s}^{2-\frac{d}{2}}\\
&\leq \frac{1}{4} \|\rho\|_{\dot{H}^{s+1}}^2+C\|\rho-\bar{\rho}\|_{L^2}^{\frac{4}{4-d}}\|\rho\|_{\dot{H}^s}^2.
\end{align*}
Second,
\begin{align*}
\bar{\rho}\|\rho\|_{\dot{H}^{s+1}}^{\frac{s}{s+1}}\|\rho-\bar{\rho}\|_{L^2}^{\frac{1}{s+1}}\|\rho\|_{\dot{H}^s}&\leq \bar{\rho}\|\rho\|_{\dot{H}^s}\|\rho\|_{\dot{H}^{s+1}}\\
&\leq \frac{1}{4}\|\rho\|_{\dot{H}^{s+1}}^2+C\bar{\rho}^2\|\rho\|_{\dot{H}^s}^2.
\end{align*}
We used the Poincare inequality and \eqref{sobeasy56} in the first step. Recall the following Nash-type inequality
\begin{equation}\label{16th}
\|\rho\|_{\dot{H}^s}\leq C\|\rho\|_{\dot{H}^{s+1}}^{\frac{2s+d}{2s+2+d}}\|\rho\|_{L^1}^{\frac{2}{2s+2+d}},
\end{equation}
the proof of which will be sketched in the appendix. Since $\rho(x,t)\geq 0$ and hence $\|\rho(\cdot,t)\|_{L^1}=\bar{\rho}>0$ is conserved in time, putting all estimates into (\ref{15th}) we get
\begin{equation}\label{17th}
\frac{1}{2}\partial_t\|\rho\|_{\dot{H}^s}^2\leq C\left(\|\rho-\bar{\rho}\|_{L^2}^{\frac{4}{4-d}}+\bar{\rho}^2+\|u\|_{C^s}\right)\|\rho\|_{\dot{H}^s}^2-c\bar{\rho}^{-\frac{4}{2s+d}}\|\rho\|_{\dot{H}^s}^{2+\frac{4}{2s+d}}.
\end{equation}
From this differential inequality and integrability of $\|\rho(\cdot, t)-\bar{\rho}\|_{L^2(\mathbb{T}^d)}^{\frac{4}{4-d}}$ in time, a finite upper bound for $\|\rho(\cdot, t)\|_{\dot{H}^s}$ follows for all times. In fact, due to the last term on the right hand side of (\ref{17th}), it is not hard to show that there is a global, not growing in time, upper bound for any $\dot{H}^s$ norm of $\rho$.
\end{proof}

\section{An $\dot{H}^1$ condition for an absorbing set in $L^2$}

Due to Theorem \ref{L2cr}, to show global regularity of the solution $\rho(x,t)$ to (\ref{chemo1}), it suffices to control its $L^2$ norm in spatial variables. In this section, we prove a simple criterion that says that if the $\dot{H}^1$ norm of a solution is sufficiently large compared to its $L^2$ norm, then in fact the $L^2$ norm is decaying. Our overall strategy will be then to show that mixing can increase and sustain the $\dot{H}^1$ norm of solution. This will block the $L^2$ norm from ever growing too much, leading to global regularity.
\begin{prop}\label{L2decay}
Let $\rho(x,t)\geq 0$ be smooth local solution to (\ref{chemo1}) set on $\mathbb{T}^d$, $d=2$ or $3$. Suppose that $\|\rho(\cdot,t)-\bar{\rho}\|_{L^2}\equiv B>0$ for some $t\geq 0$. Then there exists a universal constant $C_1$ such that
\begin{equation}\label{18th}
\|\rho(\cdot, t+\tau)-\bar{\rho}\|_{L^2}\leq 2B \mbox{ for every } 0\leq \tau \leq C_1\min(1,\bar{\rho}^{-1},B^{-\frac{4}{4-d}}).
\end{equation}
Moreover, there exists a universal constant $C_0$ such that if, in addition,
\begin{equation}\label{19th}
\|\rho(\cdot,t)\|_{\dot{H}^1}^2\geq B_1^2\equiv C_0 B^{\frac{12-2d}{4-d}}+2\bar{\rho}B^2,
\end{equation}
then $\partial_t \|\rho(\cdot, t)\|_{L^2}<0$.
\end{prop}

\begin{remark}
1. In particular, due to Theorem \ref{L2cr}, it follows that if $\|\rho(\cdot,t)-\bar{\rho}\|_{L^2}\leq B$, then the local smooth solution persists at least till $t+C_1\min\left(1,\bar{\rho}^{-1},B^{-\frac{4}{4-d}}\right)$.\\
2. Here and below, by universal constant we mean a constant that does not depend on any parameters of the problem: namely, a constant independent of $u,$ $\rho_0,$ and $B.$
\end{remark}
\begin{proof}
Let us multiply both sides of (\ref{chemo1}) by $\rho-\bar{\rho}$ and integrate. Then
\begin{equation}\label{L2}
\begin{split}
\frac12 \partial_t\|\rho-\bar{\rho}\|_{L^2}^2&=-\|\rho\|_{\dot{H}^1}^2+\int_{\mathbb{T}^d} \nabla\cdot(\rho\nabla(-\Delta)^{-1}(\rho-\bar{\rho}))(\rho-\bar{\rho})\, dx.
\end{split}
\end{equation}
Observe that
$$
\int_{\mathbb{T}^d}\rho\nabla(-\Delta)^{-1}(\rho-\bar{\rho})\nabla \rho \, dx=\int_{\mathbb{T}^d}\rho^2(\rho-\bar{\rho})\, dx-\int_{\mathbb{T}^d}\nabla \rho\cdot\nabla (-\Delta)^{-1}(\rho-\bar{\rho})\, dx.
$$
Therefore,
$$
\int_{\mathbb{T}^d}\rho\nabla(-\Delta)^{-1}(\rho-\bar{\rho})\nabla\rho \, dx=\frac{1}{2}\int_{\mathbb{T}^d}\rho^2(\rho-\bar{\rho})\, dx,
$$
and then the integral on the right hand side of (\ref{L2}) is equal to
$$
-\int_{\mathbb{T}^d}\nabla\cdot(\rho\nabla(-\Delta)^{-1}(\rho-\bar{\rho})(\rho-\bar{\rho}) )\, dx=\frac{1}{2}\int_{\mathbb{T}^d}\rho^2(\rho-\bar{\rho})\, dx,
$$
Next, notice that
$$
\int_{\mathbb{T}^d}\rho^2(\rho-\bar{\rho})\, dx=\int_{\mathbb{T}^d}(\rho-\bar{\rho})^3\, dx+2\bar{\rho}\int_{\mathbb{T}^d}(\rho-\bar{\rho})^2 \, dx-2\bar{\rho}^2.
$$
By a Gagliardo-Nirenberg inequality (see e.g. \cite{maz2013sobolev} or \cite{kiselev2012biomixing} for a simple proof), we have
$$
\|\rho-\bar{\rho}\|_{L^3}^3\leq C \|\rho-\bar{\rho}\|_{L^2}^{3-\frac{d}{2}}\|\rho\|_{\dot{H}^1}^{\frac{d}{2}}\leq \|\rho\|_{\dot{H}^1}^2+C_1\|\rho-\bar{\rho}\|_{L^2}^{\frac{12-2d}{4-d}},
$$
where in the second step we applied Young's inequality. Applying all these estimates to (\ref{L2}) yields
\begin{equation}\label{L2H1}
\partial_t\|\rho-\bar{\rho}\|_{L^2}^2\leq - \|\rho\|_{\dot{H}^1}^2+C_0\|\rho-\bar{\rho}\|_{L^2}^{\frac{12-2d}{4-d}}+2\bar{\rho}\|\rho-\bar{\rho}\|_{L^2}^2.
\end{equation}
Solving the differential equation
 $$
 f'(\tau)=Cf^{\frac{6-d}{4-d}}(\tau)+2\bar{\rho}f(\tau), \quad f(0)=B^2,
 $$
  leads to the solution
  \begin{equation}\label{22th}
  f(\tau)=\frac{B^2\exp (2\bar{\rho}\tau)}{\left(1-C\bar{\rho}^{-1}B^{\frac{4}{4-d}}(\exp(\frac{4\bar{\rho}\tau}{4-d})-1)\right)^{\frac{4-d}{2}}}.
  \end{equation}
  A standard comparison argument can be used to show that $\|\rho(\cdot, t+\tau)-\bar{\rho}\|_{L^2}^2\leq f(\tau)$. On the other hand, a straightforward estimate using (\ref{22th}) gives existence of a constant $C_1$ such that if $\tau\leq C_1 \min(1,\bar{\rho}^{-1},B^{-\frac{4}{4-d}})$, then $f(\tau)\leq 4B^2$.

  The second statement of the lemma follows directly from (\ref{L2H1}) and an assumption $\rho_0\geq 0$ (which implies $\bar \rho \geq 0$).

\end{proof}

\section{An approximation lemma}
We can now outline our general strategy of the proof of chemotactic blow up suppression in more detail. We know that the control of $L^2$ norm in spatial coordinates is sufficient for global regularity. We also see that if $\dot{H}^1$ norm of the solution is large, then $L^2$ norm is not growing. On the other hand, flows with strong mixing properties tend to increase $\dot{H}^1$ norm of solution. Hence our plan will be to deploy such flows, at a sufficiently strong intensity, to make sure that the $\dot{H}^1$ norm of the solution stays high, at least whenever the $L^2$ norm is not small. The first hurdle we face, however, is to show that the mixing property of flow persists in the full nonlinear Keller-Segel equation.

In this section we prove a key result on the approximation of solutions to the Keller-Segel equation with advection (\ref{chemo1}) by solutions of the pure advection equation. We will be looking at the intense advection regime, and consider small, relative to all parameters except the strength of advection, time intervals. It is natural to assume that in this case most of the dynamics we observe is due to advection, though the exact statement of the result requires care since both diffusion and chemotactic terms are not trivial perturbations.

 Let us consider the equation (\ref{chemo1})
 $$
 \partial_t\rho+(u\cdot\nabla)\rho-\Delta \rho+\nabla \cdot (\rho\nabla(-\Delta)^{-1}(\rho-\bar{\rho}))=0,\quad \rho(x,0)=\rho_0(x),
 $$
$x\in \mathbb{T}^d$, with $d=2,3$. We will assume that the vector field $u$ is divergence free and Lipschitz in spatial variables. It may be stationary or time dependent. Let us denote $\eta(x,t)$ the unique
Lipschitz solution of the equation
\begin{equation}\label{23th}
\partial_t\eta +(u\cdot\nabla)\eta=0, \quad \eta(x,0)=\rho_0(x).
\end{equation}
If we define the trajectories map by
\begin{equation}\label{24th}
\frac{d}{dt}\Phi_t(x)=u(\Phi_t(x),t),\quad \Phi_0(x)=x,
\end{equation}
then $\eta(x,t)=\rho_0(\Phi_t^{-1}(x))$.

We start with the following simple observation. Denote the Lipschitz semi-norm
$$
\|f\|_{Lip}=\sup_{x,y}\frac{|f(x)-f(y)|}{|x-y|}.
$$

\begin{lem}\label{51}
Suppose that the vector field $u$ is incompressible and Lipschitz in spatial variable for each $t\geq 0$, $\|u(\cdot,t)\|_{Lip(\mathbb{T}^d)}\leq D(t)$, $D(t)\in L^1_{loc}[0,\infty)$. Let $\eta(x,t)$ be the solution of (\ref{23th}). Then for every $t\geq 0$, and for every $\rho_0\in \dot{H}^1$, we have
\begin{equation}\label{25th}
\|\eta(\cdot, t)\|_{\dot{H}^1}\leq F(t)\|\rho_0\|_{\dot{H}^1}, \quad \mbox{where } F(t)=\exp\left(C\int_0^tD(s)ds\right).
\end{equation}
\end{lem}
\begin{proof}
If $u$ is incompressible and Lipschitz in spatial variable for each time, then the trajectories map $\Phi_t(x)$ is area preserving, Lipschitz in $x$ and invertible for each $t$, and the inverse map $\Phi_t^{-1}(x)$ is also Lipschitz in spatial variables. Moreover, $\|\Phi_t^{-1}\|_{Lip}\leq \exp(C\int_0^t D(s)ds)$ (see e.g. \cite{marchioro2012mathematical}). The evolution $\eta(x,t)=\rho_0(\Phi_t^{-1})$ is a Lipschitz regular coordinate change of an $\dot{H}^1$ function $\rho_0$. The bound (\ref{25th}) follows from the well known properties of $\dot{H}^1$ functions under Lipschitz transformations of coordinates \cite{ziemer1989weakly}.
\end{proof}

We are now ready to prove the approximation lemma.
\begin{lem}\label{5.2}
Suppose that the vector field $u(x,t)$ is incompressible and Lipschitz in $x$, and is such that (\ref{25th}) is satisfied with $F(t)\in L^{\infty}_{loc}[0,\infty)$. Let $\rho(x,t)$, $\eta(x,t)$ be the solutions of (\ref{chemo1}), (\ref{23th}) respectively with $\rho_0\geq 0\in \dot{H}^1$. Suppose that the unique local smooth solution $\rho(x,t)$ exists for $t\in [0,T]$. Then for every $t\in[0,T]$ we have
\begin{equation}\label{26th}
\frac{d}{dt}\|\rho-\eta\|_{L^2}^2\leq -\|\rho\|_{\dot{H}^1}^2+4\|\rho_0\|_{\dot{H}^1}^2F(t)^2+C\|\rho-\bar{\rho}\|_{L^2}^2\left(\|\rho-\bar{\rho}\|_{L^2}^{\frac{12}{6-d}}+\bar{\rho}^2\right).
\end{equation}
\end{lem}

Of course a direct analog of lemma holds without assumption $\rho_0\geq 0$; we only need to replace $\bar{\rho}^2$ in (\ref{26th}) with $\|\rho\|^2_{L^1}$.

\begin{proof}
A direct computation using divergence free property of $u$ shows that
\begin{equation}\label{27th}
\begin{split}
\frac12 \frac{d}{dt}\|\rho-\eta\|_{L^2}^2&= \int_{\mathbb{T}^d}\Delta \rho(\rho-\eta)\, dx-\int_{\mathbb{T}^d}\nabla\cdot(\rho\nabla(-\Delta)^{-1}(\rho-\bar{\rho}))(\rho-\eta)\, dx\\
&\leq -\|\rho\|^2_{\dot{H}^1}+\|\rho\|_{\dot{H}^1}\|\eta\|_{\dot{H}^1}+\|\rho\nabla(-\Delta)^{-1}(\rho-\bar{\rho})\|_{L^2}\|\rho\|_{\dot{H}^1}\\
&+\|\rho\nabla(-\Delta)^{-1}(\rho-\bar{\rho})\|_{L^2}\|\eta\|_{\dot{H}^1}.
\end{split}
\end{equation}
Applying H\"{o}lder and Gagliardo-Nirenberg inequalities, we can estimate
\begin{align*}
\|\rho\nabla(-\Delta)^{-1}(\rho-\bar{\rho})\|_{L^2}&\leq \|\rho\|_{L^3}\|\nabla(-\Delta)^{-1}(\rho-\bar{\rho})\|_{L^6}\\
&\leq C\left(\|\rho\|_{\dot{H}^1}^{\frac{d}{6}}\|\rho-\bar{\rho}\|_{L^2}^{1-\frac{d}{6}}+\bar{\rho}\right)\|\nabla(-\Delta)^{-1}(\rho-\bar{\rho})\|_{\dot{H}^1}^{\frac{d}{3}}\|\nabla(-\Delta)^{-1}(\rho-\bar{\rho})\|_{L^2}^{1-\frac{d}{3}}\\
&\leq C\left(\|\rho\|_{\dot{H}^1}^{\frac{d}{6}}\|\rho-\bar{\rho}\|_{L^2}^{1-\frac{d}{6}}+\bar{\rho}\right)\|\rho-\bar{\rho}\|_{L^2}.
\end{align*}
Here the last step follows from simple estimates on Fourier side. Given these estimates, several applications of Young's inequality show that the right hand side of (\ref{27th}) can be bounded above by
\begin{align*}
&-\|\rho\|_{\dot{H}^1}^2+\frac{1}{4}\|\rho\|_{\dot{H}^1}^2+\|\eta\|_{\dot{H}^1}^2+\frac{1}{8}\|\rho\|_{\dot{H}^1}^2+C\|\rho-\bar{\rho}\|_{L^2}^{2+\frac{12}{6-d}}\\
&+\frac{1}{8}\|\rho\|_{\dot{H}^1}^2+C\|\rho-\bar{\rho}\|_{L^2}^2\bar{\rho}^2+\|\eta\|_{\dot{H}^1}^2.
\end{align*}
With help of Lemma \ref{51}  the estimate (\ref{26th}) quickly follows.
\end{proof}

\section{Proof of the main theorem: the relaxation enhancing flows}

Our first example of flows that can stop chemotactic explosion will be relaxation enhancing flows of \cite{constantin2008diffusion}. These stationary in time flows have been shown to be very efficient in speeding up convergence to the mean of the solutions of diffusion-advection equation. A particular type of the RE flows are weakly mixing flows, a well known class in dynamical systems theory which is intermediate in mixing properties between mixing and ergodic \cite{cornfeld2012ergodic}. Let us briefly review the relevant definitions.

Given an incompressible vector field $u(x)$ which is Lipschitz in spatial variables, recall definition (\ref{24th}) for the trajectories map $\Phi_t(x)$. Then define a unitary operator $U^t f(x)=f(\Phi_t^{-1}(x))$ acting on $L^2(\mathbb{T}^d)$.

\begin{definition}
The flow $u(x)$ is called weakly mixing if the spectrum of the operator $U\equiv U^1$ is purely continuous.

The flow $u(x)$ is called relaxation enhancing (RE) if the operator $U$ (or properly defined operator ($u\cdot \nabla$)) has no eigenfunctions in $\dot{H}^1$ other than a constant function.
\end{definition}
\begin{remark*}
The fact that we talk about the spectrum of $U$ rather than $(u\cdot \nabla)$ is a minor technical point. The symmetric operator $i(u\cdot\nabla)$ is unbounded on $L^2$, and sometimes needs to be extended to its natural domain (rather than just $\dot{H}^1$) to become self-adjoint and to be a true generator for $U$. To avoid these technicalities, it is convenient to make the definition in terms of $U$, which is bounded and so can be defined on smooth functions and then extended to the entire $L^2$ by continuity.

Examples of weakly mixing flows on $\mathbb{T}^d$ are classical and go back to von Neumann \cite{neumann1932operatorenmethode} (just continuous $u(x)$) and Kolmogorov \cite{kolmogorov1953dynamical} (smooth $u(x)$). The Kolmogorov construction is based on the irrational rotation on the torus with appropriately selected invariant measure. Lack of eigenfunctions is established by analysis of a small denominator problem on the Fourier side bearing some similarity to the core of the KAM theory. The original examples are not incompressible with respect to the Lebesgue measure on $\mathbb{T}^d$, but a smooth change of coordinates can be applied to obtain incompressible flows with the same properties. Weakly mixing flows are also RE, but there are smooth RE flows which are not weakly mixing: these do have eigenfunctions but rough ones, lying in $L^2$ but not in $\dot{H}^1$ (see \cite{constantin2008diffusion} for more details and examples).
\end{remark*}
We now state our first main theorem. Consider the equation
\begin{equation}\label{28th}
\partial_t \rho^A+A(u\cdot \nabla)\rho^A -\Delta \rho^A +\nabla\cdot(\rho^A\nabla(-\Delta)^{-1}(\rho-\bar{\rho}))=0,\quad \rho^A(x,0)=\rho_0(x).
\end{equation}
Here $A$ is the coupling constant regulating strength of the fluid flow that we will assume to be large. We note that dividing the equation by $A$ and changing time, we can instead think of all the results below as applicable in the regime of weak diffusion and chemotaxis on long time scales.
\begin{thm}\label{6.2}
Suppose that $u$ is smooth and incompressible vector field on $\mathbb{T}^d$, $d=2,3$, which is also relaxation enhancing. Assume that $\rho\geq 0 \in C^{\infty}(\mathbb{T}^d)$. Then there exists an amplitude $A_0$ which depends only on $\rho_0$ and $u$ such that for every $A\geq A_0$ the solution $\rho^A(x,t)$ of the equation (\ref{28th}) is globally regular.
\end{thm}
We will only prove Theorem \ref{6.2} in the case of weakly mixing flows. This serves our main purpose of providing an example of chemotactic blow up-arresting flow. In the general RE case, the proof is a fairly straightforward extension of an argument for the weakly mixing flow and the point spectrum estimates in \cite{constantin2008diffusion} (namely, Lemma 3.3 and part of the proof of Theorem 1.4 dealing with point spectrum).

Before starting the proof, we need one auxiliary result from \cite{constantin2008diffusion}. Let $P_N$ be the orthogonal projection operator on the subspace formed by Fourier modes $|k|\leq N$:
$$
P_Nf(x)=\sum_{|k|\leq N}e^{2\pi i k x} \hat{f}(k).
$$
\begin{lem}\label{6.3}
Let $U$ be a unitary operator with purely continuous spectrum defined on $L^2(\mathbb{T}^d)$. Let $S=\{\phi\in L^2:\|\phi\|_{L^2}=1\}$, and let $K\subset S$ be a compact set. Then for every $N$ and every $\sigma>0$, there exists $T_c(N,\sigma,K,U)$ such that for all $T\geq T_c(N,\sigma,K,U)$ and every $\phi\in K$, we have
\begin{equation}\label{29th}
\frac{1}{T}\int_0^T\|P_N U^t \phi\|^2dt \leq \sigma.
\end{equation}
\end{lem}
This lemma connects the issues we are studying with one of the themes in quantum mechanics, namely the propagation rate of wave packets corresponding to continuous spectrum. Lemma \ref{6.3} is an extension of the well-known RAGE theorem (see e.g. \cite{cycon2009schrodinger}) which is a rigorous variant of a folklore quantum mechanics statement that quantum states corresponding to the continuous spectrum travel to infinity. In our case, travel to infinity happens not in physical space, but in the modes of the operator $-\Delta$ (that is, in Fourier modes). We refer to \cite{constantin2008diffusion} for the proof of Lemma \ref{6.3}.

Now we are ready to give the proof of Theorem \ref{6.2}.

\begin{proof}[Proof of Theorem \ref{6.2}]
Fix any $B> \|\rho_0-\bar{\rho}_0\|_{L^2}$. If for all times we have that $\|\rho^A(\cdot,t)-\bar{\rho}\|_{L^2}<B$, then the solution stays globally regular by Theorem \ref{L2cr}. Otherwise, let
$$
t_0=\inf\{t\big| \|\rho^A(\cdot,t)-\bar{\rho}\|_{L^2}=B\}.
$$
Since the solution is smooth up to time $t_0$, we also have that $\|\rho^A(\cdot,t_0)-\bar{\rho}\|_{L^2}=B$; thus $t_0$ is the first time the $L^2$ norm of $\rho^A(x,t)-\bar{\rho}$ reaches $B$. Note that by Proposition \ref{L2decay}, we must also have $\|\rho^A(\cdot,t_0)\|_{\dot{H}^1}<B_1$, where $B_1=C_0 B^{\frac{12-2d}{4-d}}+2\bar{\rho}B^2$.

We are going to show that if $A\geq A_0(B,\bar{\rho},u)$ is sufficiently large, then after a small time interval of length $\tau$ that we will define shortly, we will have $\|\rho^A(\cdot, t_0+\tau)-\bar{\rho}\|_{L^2}<B$. Moreover, we will have $\|\rho^A(\cdot, t)-\bar{\rho}\|_{L^2}\leq 2B$ for every $t\in [t_0,t_0+\tau]$. This will prove the theorem, as the argument can be applied repeatedly each time the $L^2$ norm reaches the level $B$, showing that $\|\rho^A(\cdot,t)-\bar{\rho}\|_{L^2}\leq 2B$ for all times.

Denote $\lambda_n$ the eigenvalues of $-\Delta$ on $\mathbb{T}^d$ in an increasing order, $0=\lambda_1\leq \lambda_2\leq \dots\leq \lambda_n\leq \dots$ Choose $N$ so that
\begin{equation}\label{30th}
\lambda_N\geq 16C_0(2B)^{\frac{4}{4-d}}+32\bar{\rho},
\end{equation}
where $C_0$ is the constant appearing in (\ref{L2H1}). Observe that $\lambda_NB^2>B_1^2$. Define the compact set $K\subset S$ by
$$
K=\{\phi\in S\big| \|\phi\|_{\dot{H}^1}^2\leq \lambda_N \}
$$
(recall $S$ is the unit sphere in $L^2$). Let $U$ be the unitary operator associated with our weakly mixing flow $u$ as above. Fix $\sigma=0.01$. Let $T_c(N,\sigma,K,U)$ be the time threshold provided by Lemma \ref{6.3}.

We proceed to impose the first condition on $A_0(\rho_0,u)$. We define $\tau$ as below and require that
\begin{equation}\label{31th}
\tau\equiv \frac{T_c(N,\sigma,K,U)}{A}\leq C_1\min\left(1,\bar{\rho}^{-1}, B^{-\frac{4}{4-d}}\right)
\end{equation}
for every $A\geq A_0$, where $C_1$ is the constant appearing in Proposition \ref{L2decay} in (\ref{18th}). It follows from Proposition \ref{L2decay} and Theorem \ref{L2cr} that $\|\rho^A(\cdot, t)-\bar{\rho})\|_{L^2}\leq 2B$ for $t\in [t_0, t_0+\tau]$ and so $\rho^A$ remains smooth on this time interval.

Let us introduce a short-cut notation $\phi_0(x)=\rho^A(x,t_0)$. Let $\eta^A(x,t)$ be the solution of the equation
$$
\partial_t\eta^A+A(u\cdot\nabla)\eta^A=0, \quad \eta^A(x,0)=\phi_0.
$$
Then $\eta^A(x,t)=U^{At}\phi_0$, and we have
\begin{equation}\label{32th}
\frac{1}{\tau}\int_0^{\tau}\|P_N\eta^A(x,t)\|_{L^2}^2dt=\frac{1}{\tau}\int_0^{\tau}\|P_NU^{At}\phi_0\|_{L^2}^2dt=\frac{1}{A\tau}\int_0^{A\tau}\|P_NU^s\phi_0\|_{L^2}^2ds\leq \sigma B^2.
\end{equation}
Here we applied Lemma \ref{6.3} to the vector $\phi_0/\|\phi_0\|_{L^2}$. Indeed, $\phi_0/\|\phi_0\|_{L^2}\in K$ since by (\ref{30th}) we have
$$
\|\phi_0\|_{\dot{H}^1}^2\leq B_1^2<\lambda_N B^2<\lambda_N\|\phi_0\|_{L^2}^2.
$$
Note also that (\ref{31th}) ensures the applicability of Lemma \ref{6.3} and the validity of the last bound in (\ref{32th}).

We now impose the second and last condition on $A_0$. It will be convenient for us now to denote by $t$ time elapsed since $t_0.$
By the approximation Lemma \ref{5.2}, and since we know that $\|\rho^A(\cdot, t_0)\|_{\dot{H}^1}^2\leq B_1<\lambda_N B^2$, as well as $\|\rho^A(\cdot, t_0+t)-\bar{\rho}\|_{L^2}\leq 2B$, we have
\begin{equation}\label{33th}
\frac{d}{dt}\|\rho^A(\cdot, t_0+t)-\eta^A(\cdot, t)\|^2_{L^2}\leq 4\lambda_N B^2F(At)^2+CB^2(B^{\frac{6}{6-d}}+\bar{\rho}^2)
\end{equation}
for all $t\in [0,\tau]$. Here $\tau=T_c/A$ as before. Choose $A_0$ so that
\begin{equation}\label{34th}
\frac{4\lambda_N}{A}\int_0^{T_0}F(s)^2ds +C\tau(B^{\frac{6}{6-d}}+\bar{\rho}^2)\leq 0.01
\end{equation}
for every $A\geq A_0$. Note that since $u$ is smooth, $F(t)$ is a locally bounded function.

We claim that if $A\geq A_0$, then $\|\rho^A(\cdot, t_0+\tau)-\bar{\rho}\|_{L^2}\leq B$. First, the condition (\ref{34th}) allows us to control $\|\rho^A(\cdot, t_0+t)-\bar{\rho}\|_{L^2}$ more tightly, which is convenient. Indeed, since $\|\eta^A(\cdot, t)\|_{L^2}=\|\phi_0\|_{L^2}=B$ for all $t\geq 0$, (\ref{33th}) and (\ref{34th}) imply that
\begin{equation}\label{35th}
0.9B\leq \|\rho^A(\cdot, t_0+t)\|_{L^2}\leq 1.1B
\end{equation}
for $t\in [0,\tau]$. Furthermore, by the estimates (\ref{32th}), (\ref{33th}) and (\ref{34th}) we have
\begin{align*}
&\frac{1}{\tau}\int_0^{\tau}\|P_N\rho^{A}(\cdot, t_0+t)\|_{L^2}^2dt\leq \frac{2}{\tau}\int_0^{\tau}\|P_N\eta^A(\cdot, t)\|_{L^2}^2dt\\
&\quad\quad+\frac{2}{\tau}\int_0^{\tau}\|P_N(\rho^A(\cdot,t_0+t)-\eta^A(\cdot,t))\|_{L^2}^2dt\leq \frac{B^2}{25}.
\end{align*}
Combining this estimate with \eqref{35th} we obtain
\begin{equation}\label{36th}
\frac{1}{\tau}\int_0^{\tau}\|\rho^A(\cdot,t_0+t)\|_{\dot{H}^1}^2dt\geq \frac{1}{\tau}\int_0^{\tau}\lambda_N\|(I-P_N)\rho^A(\cdot,t_0+t)\|_{L^2}^2dt\geq \frac{1}{2}\lambda_NB^2.
\end{equation}

Now we come back to (\ref{L2H1}):
$$
\partial_t\|\rho^A-\bar{\rho}\|_{L^2}^2\leq - \|\rho^A\|_{\dot{H}^1}^2+C_0\|\rho^A-\bar{\rho}\|_{L^2}^{\frac{12-2d}{4-d}}+2\bar{\rho}\|\rho^A-\bar{\rho}\|_{L^2}.
$$
By the estimates (\ref{36th}), (\ref{35th}) and (\ref{30th}), we have
\begin{equation}\label{37th}
\begin{split}
\|\rho^A(\cdot, t_0+\tau)-\bar{\rho}\|_{L^2}^2&\leq B^2+\tau\left(-\frac{1}{2}\lambda_NB^2+C_0(2B)^{\frac{12-2d}{4-d}}+2\bar{\rho}(2B)^2\right)\\
&\leq \left(1-\frac{1}{4}\tau\lambda_N\right)B^2\leq B^2.
\end{split}
\end{equation}
This completes the proof.
\end{proof}

We see that the bound we obtained on the decay of the $L^2$ norm in (\ref{37th}) is stronger than what we needed. In fact, with slightly more effort we can obtain stronger results. We now present an extension of Theorem \ref{6.2} that establishes a complete analog of ``relaxation enhancement'' established in \cite{constantin2008diffusion} for the diffusion-advection equation for the case that also includes chemotaxis. Namely, we show that not only fluid flow can prevent finite time blow up, but in fact it can enforce convergence of the solution to its mean in the long time limit. Intense fluid flow can also act to create an arbitrary strong and fast drop of $\|\rho^A(\cdot,t)-\bar{\rho}\|_{L^2}$.

\begin{thm}\label{6.4}
Suppose $0 \leq \rho_0 \in C^{\infty}(\mathbb{T}^d)$, and let $\rho^A(x,t)$ be the solution of the equation (\ref{28th}). Let u be smooth, incompressible, relaxation enhancing flow. If $A_0(\rho_0,u)$ is the threshold value
defined in the proof of Theorem \ref{6.2}, then for every $A\geq A_0$, we have
\begin{equation}\label{38th}
\|\rho^A(\cdot,t)-\bar{\rho}\|_{L^2}\rightarrow 0
\end{equation}
as $t\rightarrow \infty$. The convergence rate is exponential in time, and can be made arbitrary fast by increasing the value of $A$. Namely, for every $\delta>0$ and $\kappa>0$, there exists $A_1=A_1(\rho_0,u,\kappa, \delta)$ such that if $A\geq A_1$, then
\begin{equation}\label{39th}
\|\rho^A(\cdot,t)-\bar{\rho}\|_{L^2}\leq \|\rho_0-\bar{\rho}\|_{L^2}e^{-\kappa t}
\end{equation}
for all $t\geq \delta$.
\end{thm}
\begin{proof}
Let us prove (\ref{39th}), since (\ref{38th}) follows from similar (and easier) arguments. The proof largely follows the argument in the proof of Theorem \ref{6.2}, but let us outline the necessary adjustments. We set $B_0=\|\rho_0-\bar{\rho}\|_{L^2}$. 
Choose $N$ so that
\begin{equation}\label{40th}
\lambda_N\geq \max\left(100\kappa, 16C_0(2B_0)^{\frac{4}{4-d}}+32\bar{\rho}\right).
\end{equation}
Define the set $K$, as before, by $\{\phi\in S\big| \|\phi\|_{\dot{H}^1}^2\leq \lambda_N\}$.

For all times $t$ where $\rho^A(x,t)/\|\rho^A(\cdot,t)\|_{L^2}\notin K$, we have $\|\rho^A(\cdot, t)\|_{\dot{H}^1}\geq \lambda_N\|\rho^A(\cdot,t)-\bar{\rho}\|_{L^2}$. It follows from (\ref{L2H1}) and (\ref{40th}) that at such times we have
\begin{equation}\label{41th}
\begin{split}
\partial_t\|\rho-\bar{\rho}\|_{L^2}^2&\leq - \|\rho^A\|_{\dot{H}^1}^2+C_0\|\rho^A-\bar{\rho}\|_{L^2}^{\frac{12-2d}{4-d}}+2\bar{\rho}\|\rho^A-\bar{\rho}\|_{L^2}\\
&\leq \left(-\lambda_N+C_0(2B_0)^{\frac{4}{4-d}}+2\bar{\rho}\right)\|\rho^A-\bar{\rho}\|_{L^2}^2\leq -\frac{1}{2}\lambda_N\|\rho^A-\bar{\rho}\|_{L^2}^2.
\end{split}
\end{equation}
Here in the second inequality we used that $\|\rho^A(\cdot,t)-\bar{\rho}\|_{L^2}\leq 2B_0$ for all times, as we know from the proof of Theorem \ref{6.2} we can ensure by making $A$ sufficiently large; we note that this bound will also follow from our argument below. Thus on the time intervals where $\rho^A(x,t)/\|\rho^A(x,t)\|_{L^2}\notin K$ we have exponential decay of $\|\rho^A(\cdot, t)-\bar{\rho}\|_{L^2}$ at  rate that would imply (\ref{39th}) if all times were like that.

Suppose now that $t_0$ is the smallest time such that $\rho^A(x,t_0)\in K$ ($t_0$ could equal $0$). Let $T_c(N,\sigma,K,U)$ be the time threshold provided by Lemma \ref{6.3} (we set $\sigma=0.01$ as before). Repeat all the steps in the proof of Theorem \ref{6.2} from defining the time step $\tau$ (\ref{31th}) to (\ref{37th}), with $B$ replaced by $\|\rho^A(\cdot, t_0)-\bar{\rho}\|_{L^2}$. In addition, require that $A$ is large enough so that
\begin{equation}\label{42th}
\tau=\frac{T_c(N,\sigma,K,U)}{A}\leq \delta/2.
\end{equation}
We arrive at the estimate
\begin{equation}\label{43th}
\begin{split}
\|\rho^A(\cdot, t_0+\tau)-\bar{\rho}\|_{L^2}&\leq \left(1-\frac{1}{4}\lambda_N\tau \right)\|\rho^A(\cdot, t_0)-\bar{\rho}\|_{L^2}^2\\
&\leq e^{-\frac{1}{4}\lambda_N\tau}\|\rho^A(\cdot, t_0)-\bar{\rho}\|_{L^2}^2.
\end{split}
\end{equation}
Note that even though we do not control the $L^2$ norm of the solution for all times $t$, from (\ref{L2H1}) and (\ref{40th}) it is clear that for every $t\in [t_0,t_0+\tau]\equiv I_0$, we have
\begin{equation}\label{44th}
\|\rho^A(\cdot, t)-\bar{\rho}\|_{L^2}^2\leq e^{\frac{1}{8}\lambda_N (t-t_0)}\|\rho^A(\cdot, t_0)-\bar{\rho}\|_{L^2}^2.
\end{equation}

We continue further in time in a similar fashion. If $\rho^A(x,t)/\|\rho^{A}(\cdot,t)\|_{L^2}\notin K$, we have (\ref{41th}). On the other hand, if
$$
t_n =\inf \{t\big| t\geq t_{n-1}+\tau, \rho^A(x,t)/\|\rho^{A}(\cdot,t)\|_{L^2}\in K\},
$$
we can apply Lemma \ref{6.3} and Lemma \ref{5.2} on $I_n\equiv [t_n, t_n+\tau]$ obtaining
\begin{equation}\label{45th}
\begin{split}
&\|\rho^A(\cdot, t_n+\tau)-\bar{\rho}\|_{L^2}^2\leq e^{-\frac{1}{4}\lambda_N \tau}\|\rho^A(\cdot, t_n)-\bar{\rho}\|_{L^2}^2,\\
&\|\rho^A(\cdot,t)-\bar{\rho}\|_{L^2}^2\leq e^{\frac{1}{8}\lambda_N(t-t_n)}\|\rho^A(\cdot,t_n)-\bar{\rho}\|_{L^2}^2,\quad \mbox{ for every } t\in [t_n,t_n+\tau].
\end{split}
\end{equation}
Now given any $t\geq \delta$, we can represent
$$
[0,t]=W\cup (\cup_{l=0}^n I_l),
$$
where $W$ is the set of times in $[0,t]$ outside all $I_l$. Note that (\ref{41th}) holds for every $s\in W$. Combining (\ref{41th}) and (\ref{45th}), we infer that for every $t\geq \delta$, we have
$$
\|\rho^A(\cdot, t)-\bar{\rho}\|_{L^2}^2\leq e^{-\frac{1}{4}\lambda_N(t-\frac{\delta}{2})}e^{\frac{1}{8}\lambda_N\frac{\delta}{2}}\|\rho_0-\bar{\rho}\|_{L^2}^2\leq e^{-\frac{1}{8}\lambda_N t}\|\rho_0-\bar{\rho}\|_{L^2}^2.
$$
This proves (\ref{39th}).
\end{proof}

\section{Yao-Zlatos flow}

In this section, we describe another class of flows that are capable of suppressing the chemotactic explosion. These flows arise as examples of almost perfect mixers satisfying some sort of natural constraints. They have the advantage of being somewhat more explicitly defined than the RE flows which may be harder to picture. For this reason we will be able to get an explicit estimate, albeit rather weak, for the intensity of mixing necessary to arrest the blow up as a function of the $L^2$ norm of the initial data. For the RE flows, such estimate would be difficult to obtain, primarily due to the challenge of estimating the time $T_c$ from Lemma \ref{6.3}. A quantitative estimate on $T_c$ would require delicate spectral analysis of the operator $u\cdot \nabla$, something that for the moment is out of reach. On the other hand, in contrast to the RE flows, Yao-Zlatos flows are time dependent and active - their construction depends on the density being mixed. We remark that a re
 lated class of efficient mixer flows has been also considered in \cite{alberti2014exponential}.

 We refer to \cite{yao2014mixing} for a detailed discussion of different notions of mixing and the general background, and for further references. For our purpose here, we need one particular result from \cite{yao2014mixing}, that we set about to explain. Let $\mathbb{T}^2\equiv [-1/2,1/2)^2$. Consider the dyadic partition of $\mathbb{T}^2$ with squares $Q_{nij}$ given by
 \begin{equation}\label{46th}
 Q_{nij}=\left[\frac{i}{2^n},\frac{i+1}{2^n}\right]\times\left[\frac{j}{2^n},\frac{j+1}{2^n}\right],\quad i,j=-2^{n-1},\dots,2^{n-1}-1.
 \end{equation}
Suppose $f_0\in C^{\infty}(\mathbb{T}^2)$ and is mean zero, and $u$ is an incompressible flow which is Lipschitz regular in spatial variables. Let $f(x,t)$ denote the solution of transport equation
\begin{equation}\label{47th}
\partial_t f+(u\cdot \nabla)f=0, \quad f(x,0)=f_0(x).
\end{equation}
\begin{thm}\label{7.1}[Yao-Zlatos]
Given any $\kappa, \epsilon\in (0,1/2]$, there exists an incompressible flow $u$ such that the following holds.
\begin{equation}\label{48th}
1.\quad \|\nabla u(\cdot,t)\|_{L^{\infty}}\leq 1, \quad \mbox{ for every }t.
\end{equation}
2. Let $n=[|\log_2(\kappa\epsilon)|]+2$, where $[x]$ denotes the integer part of $x$. Then for some
$$
\tau_{\kappa,\epsilon}\leq C\kappa^{-1/2}|\log(\kappa\epsilon)|^{3/2},
$$
and every $Q_{nij}$ as in (\ref{46th}) we have
\begin{equation}\label{49th}
\left|\frac{1}{|Q_{nij}|}\int_{Q_{nij}}f(x,\tau_{\kappa,\epsilon})\,\, dx\right|\leq \kappa \|f_0\|_{L^{\infty}}.
\end{equation}
\end{thm}

This theorem provides a flow $u$ that satisfies uniform in time Lipschitz constraint (\ref{48th}) and mixes the initial density $f_0$ to scale $\epsilon$ with the error at most $\kappa$ in time $\tau_{\kappa,\epsilon}$. The construction in \cite{yao2014mixing} employs a multi-scale cellular flow. On the $n$th stage of the construction, the goal is to make the mean of the function on each of $Q_{nij}$ close to zero, starting with $n=1$. This is achieved by cellular flows which are designed to have the same rotation time period on streamlines away from a thin boundary layer. Such an arrangement makes evolution of density in each cell amenable to fairly precise control, and makes it possible to ensure that the ``nearly mean zero'' property of the solution propagates to smaller and smaller scales. We refer to \cite{yao2014mixing} for the details. Below, we outline only some adjustments that are needed to obtain Theorem~\ref{7.1} from the
arguments in \cite{yao2014mixing}, since it is not stated there in the precise form that we need.

\begin{proof}
Theorem \ref{7.1} is essentially Theorem 4.3 from \cite{yao2014mixing} (or rather Theorem 5.1, which deals with the periodic boundary conditions instead of no flow - but Theorem 5.1 is a direct corollary of Theorem 4.3). We re-scaled time compared to Theorem 4.3 from \cite{yao2014mixing} to make (\ref{48th}) hold. We also replaced the conclusion of the $\epsilon$-scale mixing (defined in \cite{yao2014mixing}) with how it is actually proved: (\ref{49th}) follows directly from (4.8) and the next estimate in \cite{yao2014mixing} as well as the choice of $\delta$ immediately below these two estimates.
\end{proof}

Here is the key corollary of Theorem \ref{7.1} that we will use in our proof.

\begin{cor}\label{7.2}
Let $f_0\in C^{\infty}(\mathbb{T}^2)$ be a mean zero function. For every $\epsilon>0$ there exists a Yao-Zlatos flow $u(x,t)$ given by Theorem~\ref{7.1}, such that $\|\nabla u\|_{L^{\infty}}\leq 1$ for every $t$ and the solution $f(x,t)$ of the equation \eqref{47th} satisfies
\begin{equation}\label{50th}
\|f(\cdot, \tau)\|_{\dot{H}^{-1}}\leq C_3\|f_0\|_{L^{\infty}}\epsilon
\end{equation}
for some
\begin{equation}\label{51th}
\tau\leq C_2\epsilon^{-1/2}|\log\epsilon|^{3/2}.
\end{equation}
Here $C_{2,3} \geq 1$ are universal constants.
\end{cor}

\begin{proof}
To derive this corollary from Theorem \ref{7.1}, let us set $\kappa=\epsilon$. We need to address a couple of issues. The first one is the connection between (\ref{49th}) and $\dot{H}^{-1}$ norm of the solution.

\begin{lem}\label{7.3}
Let $f\in C^{\infty}(\mathbb{T}^2)$ be mean zero. Fix $\epsilon>0$ and suppose that
$$
\left|\frac{1}{|Q_{nij}|}\int_{Q_{nij}}f(x)\, dx\right|\leq \epsilon \|f\|_{L^{\infty}}
$$
for some $n\geq [|\log_2\epsilon|]$ and $i,j=-2^{n-1}, \dots, 2^{n-1}-1$. Then
\begin{equation}\label{52th}
\|f\|_{\dot{H}^{-1}}\leq C_3\|f\|_{L^{\infty}}\epsilon.
\end{equation}
\end{lem}
\begin{proof}
The proof is by duality. Take any $g\in \dot{H}^1$. Since $f$ is mean zero, without loss of generality we can assume that $g$ is also mean zero. Then, denoting $\bar{g}_{Q_{nij}}$ the average value of $g$ over $Q_{nij}$, we have
\begin{align*}
&\left|\int_{\mathbb{T}^2}fg\, dx\right|=\left|\sum_{i,j}\int_{Q_{nij}}fg\, dx\right|\leq \left|\sum_{i,j}\int_{Q_{nij}}f(x)(g(x)-\bar{g}_{Q_{nij}})\, dx\right|+\\ &\left|\sum_{i,j}\bar{g}_{Q_{nij}}\int_{Q_{nij}}f(x)\, dx\right| \leq \sum_{i,j}\|f\|_{L^2(Q_{nij})}\|g-\bar{g}_{Q_{nij}}\|_{L^2(Q_{nij})}+\\
&\epsilon \sum_{i,j}|\bar{g}_{Q_{nij}}| |Q_{nij}|\|f\|_{L^{\infty}}\leq C2^{-n}\sum_{i,j}\|f\|_{L^2(Q_{nij})}\|\nabla g\|_{L^2(Q_{nij})}+\epsilon \|f\|_{L^{\infty}}\|g\|_{L^1}\\
&\leq C\epsilon \|f\|_{L^2}\|g\|_{\dot{H}^1}+\epsilon \|f\|_{L^{\infty}}\|g\|_{\dot{H}^1} \leq C\epsilon \|f\|_{L^{\infty}}\|g\|_{\dot{H}^1}.
\end{align*}
Here we used Poincare inequality in the last and in the penultimate step, and Cauchy-Schwartz inequality in the last step. This proves the lemma.
\end{proof}

 Another technical aspect we need to discuss is the smoothness of $u$. The construction in \cite{yao2014mixing} does not explicitly control higher order derivatives of $u$ beyond the Lipschitz condition $\|\nabla u\|_{L^{\infty}}\leq 1$. However, it is not difficult to see that a properly mollified velocity field will have the same mixing properties up to renormalization by some universal constant. Let $\phi$ be a mollifier, $\phi\geq 0$, $\phi\in C^{\infty}(\mathbb{T}^2)$, supp$(\phi)\subset B_{1/4}(0)$, $\int_{\mathbb{T}^2}\phi (x)\, dx=1$. Denote $\phi_{\delta}(x)= \delta^{-d}\phi(x/\delta)$, and $u_{\delta}(x)=\phi_{\delta}\ast u(x)$.
\begin{lem}\label{7.4}
Suppose that an incompressible vector field $u(x,t)$ satisfies $\|\nabla u\|_{L^{\infty}}\leq D$ for all $x,t$. Assume $f_0\in C^{\infty}(\mathbb{T}^2)$, and denote $f(x,t)$ and $f_{\delta}(x,t)$ solutions of the transport equation (\ref{47th}) with velocity $u$ and mollified velocity $u_{\delta}$ respectively. Fix any $T>0$. Then as $\delta\rightarrow 0$, we have $\|f(x,t)-f_{\delta}(x,t)\|_{L^2}\rightarrow 0$ uniformly in $t\in [0,T]$.
\end{lem}
\begin{proof}
The proof of this lemma is standard and elementary. We provide a brief sketch. First note that $\|\nabla u_{\delta}(\cdot, t)\|_{L^{\infty}}\leq \|\nabla u(\cdot, t)\|_{L^{\infty}}$.
Consider $\Phi_t(x)$ and $\Phi_{t,\delta}(x)$, the trajectory maps corresponding to $u$ and $u_{\delta}$. A straightforward estimate based on Gronwall lemma gives
$$
|\Phi_t(x)-\Phi_{t,\delta}(x)|\leq \delta e^{tD}.
$$
Reversing time, we find that the same holds for the inverse maps:
$$
|\Phi^{-1}_t(x)-\Phi_{t,\delta}^{-1}(x)|\leq \delta e^{tD}.
$$
So we get
$$
\int_{\mathbb{T}^2}|f(x,t)-f_{\delta}(x,t)|^2\, dx=\int_{\mathbb{T}^2}|f_0(\Phi_t^{-1}(x))-f_0(\Phi_{t,\delta}^{-1}(x))|^2\, dx\leq \delta\|\nabla f_0\|_{L^{\infty}}^2e^{tD}.
$$
\end{proof}

To complete the proof of the corollary, note now that we can choose $\delta = \delta(f_0)$ small enough so that
$$
\|f_{\delta}(\cdot,\tau)-f(\cdot,\tau)\|_{\dot{H}^{-1}}\leq \|f_{\delta}(\cdot,\tau)-f(\cdot,\tau)\|_{L^{2}}\leq C_3\|f_0\|_{L^{\infty}}\epsilon.
$$
Then if $u(x,t)$ is Yao-Zlatos vector field yielding (\ref{50th}), we can take $u_{\delta}$ as our smooth flow and get that (\ref{50th}) holds with the renormalized constant $2C_3$.
\end{proof}

Before stating our main result on Yao-Zlatos flows, we need one more auxiliary result. In our scheme, it is convenient to work with the $L^2$ norm of the solution. However Corollary \ref{7.2} involves the $L^{\infty}$ norm, so we need some control over it. We could get it from (\ref{17th}) and Sobolev embedding. However the bound in \eqref{17th} involves norms of higher order derivatives of $u$ and would result in weaker estimates for the flow intensity needed to suppress blow up. We prefer to estimate the $L^{\infty}$ norm of the solution directly.

Let $\rho(x,t)$ be the solution of (\ref{chemo1}):
$$
\partial_t\rho+(u\cdot\nabla) \rho -\Delta \rho + \nabla\cdot(\rho \nabla (-\Delta)^{-1}(\rho-\bar{\rho}))=0,\quad \rho(x,0)=\rho_0(x),\quad x\in \mathbb{T}^d.
$$
where $u(x,t)$ is smooth and incompressible.
\begin{prop}\label{7.5}
Let $0 \leq \rho\in C^{\infty}(\mathbb{T}^2)$. Suppose that $\|\rho(\cdot, t)-\bar{\rho}\|_{L^2}\leq 2B$ for all $t\in [0,T]$ and some $B\geq 1$. Then we also have $\|\rho(\cdot,t)-\bar{\rho}\|_{L^{\infty}}\leq C_4 B\max(B,\bar{\rho}^{1/2})$ for some universal constant $C_4$ and all time $t\in [0,T]$.
\end{prop}

We postpone the proof of this proposition to the appendix.

We are now ready to state the main theorem of this section.

\begin{thm}
Let $\rho_0\geq 0\in C^{\infty}(\mathbb{T}^2)$, and suppose $\|\rho_0-\bar{\rho}\|_{L^2}<B$ for some $B>1$.
Then there exists smooth incompressible flow $u(x,t)$ with $\|\nabla u(\cdot, t)\|_{L^\infty} \leq A(B,\bar \rho)$
such that the solution $\rho(x,t)$ of the equation \eqref{chemo1} is globally regular.
Here we can choose
\begin{equation}\label{53th}
A = C\exp\left(C(1+B+\bar{\rho}^{1/2})\big(\log(1+B+\bar{\rho}^{1/2})\big)^{3/2}\right)
\end{equation}
for some universal constant $C$.

The flow $u(x,t)$ can be represented as
\begin{equation}\label{mixflowmultyz} u(x,t) = \sum_j A u_j(x,t) \chi_{I_j}(t), \end{equation}
where $I_j$ are disjoint time intervals, and $u_j$ are Yao-Zlatos flows given by Corollary~\ref{7.2} with a certain $\epsilon = \epsilon(B,\bar \rho)>0$
and certain initial data.
\end{thm}
\begin{remark}
We have to deploy different Yao-Zlatos flows in \eqref{mixflowmultyz} due to the fact that these flows are designed to mix a specific
initial data, and one can envision nonlinear dynamics attempting to break the $L^2$ barrier in different ways.
\end{remark}
\begin{remark}
Similarly to Theorem~\ref{6.4} for the RE flows, one can use combinations of Yao-Zlatos flows to achieve stronger results, such as convergence of the solution to the mean, and at increasingly fast rate if the flow amplitude is allowed to grow. We will not pursue a walk through these results here, instead leaving it to the interested reader.
\end{remark}
\begin{proof}
The scheme of the proof is similar to the RE flows case, but Corollary \ref{7.2} replaces Lemma \ref{6.3} with some necessary adjustments.

We start with $u=0$ on some initial time interval.
If $\|\rho(\cdot, t)-\bar{\rho}\|_{L^2}\leq B$ for all $t$, global regularity follows. Suppose this is not the case, and let $t_0$ be the first time when $\|\rho(\cdot, t_0)-\bar{\rho}\|_{L^2}=B$. Similarly to the RE case, we also know that $\|\rho(\cdot,t_0)\|_{\dot{H}^1}\leq B_1$. In addition, Proposition \ref{7.5} ensures that $\|\rho(\cdot, t_0)-\bar{\rho}\|_{L^{\infty}}\leq C_4 B \max(B,\bar{\rho}^{1/2})$.

Fix $\epsilon$ given by
\begin{equation}\label{54th}
\epsilon = \frac{1}{8C_3\sqrt{4C_0B^2+2\bar{\rho}+1}(1+C_4\max(B,\bar{\rho}^{1/2}))}.
\end{equation}
Take a flow $u$ guaranteed by Corollary \ref{7.2} corresponding to this value of $\epsilon$ and the initial density $\rho(x,t_0)$, and let $\tau$ be the time in \eqref{51th}.
Denote $\eta^A(x,t)$ the solution of the equation
\begin{equation}\label{freeetaA}
\partial_t\eta^A+A(u\cdot \nabla)\eta^A=0, \quad \eta^A(x,0)=\eta_0\equiv \rho(x,t_0).
\end{equation}
Then by Corollary \ref{7.2} and Proposition \ref{7.5} we have
\begin{equation}\label{55th}
\|\eta^A(\cdot, \tau/A)-\bar{\rho}\|_{\dot{H}^{-1}}\leq C_3\|\eta_0-\bar{\rho}\|_{L^{\infty}}\epsilon\leq \frac{B}{8\sqrt{4C_0B^2+2\bar{\rho}+1}}.
\end{equation}
We will set $u(x,t)=0$ in \eqref{freeetaA} for $t\geq \tau/A$.

Now we are going to turn on the same flow $u(x,t)$ in the equation for $\rho^A$ at time $t_0$, for the duration $\tau/A$. Let us denote $\rho^A(x,t_0+t)$ the solution of the equation
$$
\partial_t\rho^A+A(u\cdot\nabla) \rho^A -\Delta \rho^A + \nabla\cdot(\rho^A \nabla (-\Delta)^{-1}(\rho^A-\bar{\rho}))=0,\quad \rho^A(x,0)=\rho(x,t_0).
$$
The first condition that we are going to impose on $A$ is that
\begin{equation}\label{56th}
\frac{2\tau}{A}\leq C_1\min\big(1,\bar{\rho}^{-1},B^{-2}\big).
\end{equation}
Given \eqref{51th} and \eqref{54th}, it is easy to check that \eqref{56th} holds if \eqref{53th} is satisfied.
Recall that given \eqref{56th}, Proposition \ref{L2decay} ensures
\begin{equation}\label{57th}
\|\rho^A(x,t_0+t)-\bar{\rho}\|_{L^2} \leq 2B
\end{equation}
for $t\in [0,2\tau/A]$ and the solution stays smooth on this time interval. Next, by the approximation Lemma \ref{5.2}, we have
\begin{equation}\label{58th}
\begin{split}
\frac{d}{dt}\|\rho^A(\cdot, t_0+t)-\eta^A(\cdot, t)\|_{L^2}^2&\leq -\|\rho^A(\cdot,t_0+t)\|_{\dot{H}^1}^2+4\|\rho(\cdot,  t_0)\|_{\dot{H}^1}^2\exp\left(2C\int_0^{At}\|\nabla u\|_{L^{\infty}}ds\right)\\
&+C\|\rho^A(\cdot, t_0+t)-\bar{\rho}\|_{L^2}^2\left(\|\rho^A(\cdot, t_0+t)-\bar{\rho}\|_{L^2}^{3}+\bar{\rho}^2\right)\\
&\leq 4B_1^2\exp(2CAt)+4CB^2\left((2B)^{3}+\bar{\rho}^2\right)
\end{split}
\end{equation}
for every $t\in [0, 2\tau/A]$. Here we used (\ref{57th}) in the last step. Let us now impose the second condition on $A$ which says that it should be large enough so that
\begin{equation}\label{59th}
\frac{2B_1^2}{CA}\exp(4C\tau)+\frac{2\tau}{A}4CB^2\left((2B)^{3}+\bar{\rho}^2\right)\leq \frac{B^2}{64(4C_0B^2+2\bar{\rho}+1)}.
\end{equation}
Note that in particular (\ref{59th}) implies that
\begin{equation}\label{60th}
0.8B\leq \|\rho^A(\cdot, t_0+t)-\bar{\rho}\|_{L^2}\leq 1.2B
\end{equation}
for $t\in [0,2\tau/A]$. Also, (\ref{59th}), (\ref{55th}) and (\ref{58th}) can be used to estimate that for every $t \in [\tau/A, 2\tau/A]$ we have
\begin{align*}
\|\rho^A(\cdot, t_0+t)-\bar{\rho}\|_{\dot{H}^{-1}} \leq \|\eta^A(\cdot, t) - \bar \rho\|_{\dot{H}^{-1}} + \|\rho^A(\cdot, t_0+t)-\eta^A(\cdot, t)\|_{L^2} \leq \\
\frac{B}{8\sqrt{4C_0B^2+2\bar{\rho}+1}}+\frac{B}{8\sqrt{4C_0B^2+2\bar{\rho}+1}}=\frac{B}{4\sqrt{4C_0B^2+2\bar{\rho}+1}}.
\end{align*}
Therefore, using \eqref{60th}, we obtain
\begin{equation}\label{61th}
\|\rho^A(\cdot, t_0+t)-\bar{\rho}\|_{\dot{H}^1}\geq \frac{\|\rho^A(\cdot, t_0+t)-\bar{\rho}\|_{L^2}^2}{\|\rho^A(\cdot, t_0+t)-\bar{\rho}\|_{\dot{H}^{-1}}}\geq 2B\sqrt{4C_0B^2+2\bar{\rho}+1}
\end{equation}
for all $t\in [\tau/A, 2\tau/A]$. The proof of global regularity is completed similarly to the proof of Theorem \ref{6.2} using (\ref{L2H1}), (\ref{60th}) and (\ref{61th}).

Finally, the sufficiency of the condition (\ref{53th}) follows from straightforward analysis of the bounds (\ref{54th}), (\ref{51th}), and (\ref{59th}).
The approximation lemma constraint (\ref{59th}) is what truly determines the exponential form of (\ref{53th}).
\end{proof}
\begin{remark}\label{noslip}
One can check that by using Theorem 5.4 from \cite{yao2014mixing} and similar arguments, we can change the periodic setting to the finite domain with no-slip or no-flow boundary condition for $u(x,t)$, and obtain analogous results but with a somewhat weaker estimate for the flow intensity. 
We leave the details to the interested reader.
\end{remark}

\section{Discussion and generalization}

The scheme developed in the previous sections of this paper should be flexible enough to be applied in different situations. Here we briefly and informally discuss the main features that appear to be necessary to apply our analysis. On the most general informal level, one can say that the idea of the scheme is that the fluid flow, if sufficiently intense and with strong
mixing properties, can make a supercritical equation into subcritical one for a given initial data.

It appears that for the scheme to be applicable to a nonlinearity $N(\rho)$ we need that for the solutions of the equation
$$
\partial_t\rho+(u\cdot \nabla)\rho-\Delta \rho+N(\rho)=0,
$$
either the mean or some norm of $\rho$ does not grow or at least obeys global finite (even if growing) bound in time. Without such assumption, it is difficult to rule our finite time blow up of the mean value of the solution, which large $\dot{H}^1$ norm has no way to arrest. The second condition that is needed concerns the bound on the nonlinear term in the spirit of
\begin{equation}\label{62th}
\left|\int N(\rho)\rho \, dx\right|\leq Cf(\|\rho\|_{L^2})\|\rho\|_{\dot{H}^1}^a,
\end{equation}
with $a<2$. This would allow control of the $L^2$ norm growth by diffusion when $\dot{H}^1$ norm is large. The third condition is that some analog of the approximation lemma holds. This seems to require bounds similar to (\ref{62th}).


Of course, the scheme can also be adapted to the cases where diffusion is given by some dissipative operator other than Laplacian, for example
a sufficiently strong fractional Laplacian, in which case the $\dot{H}^1$ norm needs to be replaced by some other norm natural in the given context. It is also likely that diffusion term
does not have to be linear, even though this may require subtler analysis.

As far as other possible classes of flows that may have the blow up arresting property, the main clearly sufficient requirement for our scheme to be applicable appears to be as follows. First, the flows should be sufficiently regular and in particular satisfy Lipschitz bound in spatial variables. Secondly, for every $\epsilon > 0$ it should be possible to find a flow $u_{\epsilon}$ from the given class, with uniform in time Lipschitz bound, such that for every $f_0\in C^{\infty}$ the solution $f(x,t)$ of the transport equation
$$
\partial_t f+(u_\epsilon \cdot \nabla)f=0, \quad f(x,0)=f_0(x)
$$
satisfies
\begin{equation}\label{63th}
\|f(\cdot, \tau_{\epsilon})\|_{\dot{H}^{-1}}\leq \epsilon C(\|f_0\|_{L^\infty})
\end{equation}
for some $\tau_{\epsilon}<\infty$. Observe that even though we did not frame the discussion of the mixing effect of the RE flows in terms of the $\dot{H}^{-1}$ norm, Lemma \ref{6.3} clearly implies that (\ref{63th}) holds for the RE flows. There are other classes of flows that look likely to satisfy these properties, such as optimal mixing flows discussed in \cite{lin2011optimal}. Generally,
decay of the $\dot{H}^{-1}$ norm is one of the general measures of the mixing ability of the flow, hence we have a link between efficient mixing and suppression of blow up, which is quite
natural. We refer to \cite{lin2011optimal}, \cite{lunasin2012optimal}, \cite{iyer2014lower}, \cite{seis2013maximal} for further discussion of $\dot{H}^{-1}$ norm as a measure of
mixing and some bounds on mixing rates for natural classes of flows. It also looks possible that some flows that do not in general lead to $\dot{H}^{-1}$ norm decay without diffusion can still be effective suppressors of blow up
if diffusion is taken into account. A natural and common class to be investigated here are some families of stationary cellular flows.
Furthermore, similarly to \cite{constantin2008diffusion}, our construction can be applied more generally to the case where the transport part of the equation is replaced by some other unitary evolution for which an analog of (\ref{63th}) holds. We plan to address some of these generalizations in future work.

\section {Appendix I: Finite time blow up}\label{finbusec}

In this section, our main result is a construction of examples where solutions to the Keller-Segel equation set on $\Tm^2$ without advection (\ref{chemo}) blow up in finite time. As we mentioned in the introduction, similar results are well known in slightly different settings. The argument below is included for the sake of completeness. It is closely related to the construction of \cite{nagai2001blowup}, but is simpler. The argument is essentially local and can be adapted to other situations as well. Although we will focus on the $d=2$ case, some auxiliary results that remain valid in every dimension will be presented in more generality.

\begin{thm}\label{blowupt}
There exist $\rho_0\in C^{\infty}(\mathbb{T}^2)$, $\rho_0\geq 0$, such that the corresponding solution $\rho(x,t)$ of the equation (\ref{chemo}) set on $\mathbb{T}^2$ blows up in finite time.
\end{thm}
Without loss of generality, we will assume that the spatial period of initial data and solution is equal to one, so $\mathbb{T}^2=[-1/2,1.2]^2$. Let us first state a lemma that will allow us to conveniently estimate the chemotactic term in the equation.
\begin{lem}\label{2.2}
Assume $\mathbb{T}^d=[-1/2,1/2]^d$, $d\geq 2$. For every $f(x)\in C^{\infty}(\mathbb{T}^d)$, we have
\begin{equation}\label{infinsum}
\nabla(-\Delta)^{-1}(f(x)-\bar{f})=-\frac{1}{c_d}\lim_{\gamma\rightarrow 0+}\int_{\mathbb{R}^d}\frac{(x-y)}{|x-y|^d}(f(y)-\bar{f})e^{-\gamma |y|^2}dy.
\end{equation}
Here on the right hand side $f(y)$ is extended periodically to all $\mathbb{R}^d$, $\bar f$ denotes the mean value of $f$,
and $c_d$ is the area of the unit sphere in $d$ dimensions.
\end{lem}

The expression (\ref{infinsum}) is of course valid for a broader class of $f$, but the stated result is sufficient for our purpose.

\begin{proof}
Without loss of generality, let us assume that $f$ is mean zero. By definition and properties of Fourier transformation, we have
$$
\nabla(-\Delta)^{-1}f(x)=-\sum_{k\in\mathbb{Z}^d,k\neq 0}e^{2\pi i k x}\frac{ik}{2\pi |k|^2}\hat{f}(k).
$$
To link this expression with (\ref{infinsum}), observe first that for a smooth $f$, a straightforward computation shows that
\begin{equation}\label{4}
-\sum_{k\in\mathbb{Z}^d, k\neq 0} e^{2\pi i k x}\frac{ik}{2\pi |k|^2}\hat{f}(k)=-\lim_{\gamma\rightarrow 0+}\int_{\mathbb{R}^d}e^{2\pi i p x}\frac{ip}{2\pi |p|^2}\int_{\mathbb{R}^d}e^{-2\pi i p y-\gamma |y|^2}f(y) dydp,
\end{equation}
where the function $f(y)$ is extended periodically to the whole plane. Indeed, all we need to do is plug in the Fourier expansion $f(y)=\sum_{k\in \mathbb{Z}^d}e^{2\pi i k y}\hat{f}(k)$, integrate in $y$, and observe that $(\pi/\gamma)^{d/2}\exp(-\pi^2|k-p|/\gamma)$ is an approximation of identity.

On the other hand, recall that the inverse Laplacian $(-\Delta)^{-1}g$ of a sufficiently regular and rapidly decaying function $g$ is given by
\begin{equation}\label{5th}
\int_{\mathbb{R}^d}e^{2\pi i p x}\frac{1}{(2\pi|p|)^2}\int_{\mathbb{R}^d}e^{-2\pi i p y}g(y)\,dydp=
\begin{cases}
\hfill -\frac{1}{2\pi}\int_{\mathbb{R}^d}\log|x-y|g(y)dy\hfill & d=2;\\
\hfill \frac{1}{c_d}\int_{\mathbb{R}^d}|x-y|^{2-d}g(y)dy\hfill & d\geq 3.
\end{cases}
\end{equation}

The expression on the right hand side of (\ref{4}), with help of (\ref{5th}), can be written as

\[
\mbox{Right hand side of }(\ref{4})=\lim_{\gamma\rightarrow 0+}\left\{
\begin{array}{ll}
\frac{1}{2\pi}\int_{\mathbb{R}^2}\log|x-y|\nabla \left(f(y)e^{-\gamma|y|^2}\right)dy\quad \ \ d=2;\\
-\frac{1}{c_d}\int_{\mathbb{R}^d}|x-y|^{2-d}\nabla \left(f(y)e^{-\gamma|y|^2}\right)dy\quad d\geq 3.
\end{array}
\right.
\]
Integrating by parts, we obtain (\ref{infinsum}).
\end{proof}

Suppose that the initial data $\rho_0$ is concentrated in a small ball $B_a$ of radius $a$ centered at the origin, so that $\int_{B_a}\rho_0\, dx=\int_{\mathbb{T}^2}\rho_0\, dx\equiv M$. Suppose that $\frac{1}{4}>b>2a$, $M>1$, and let $\phi$ be a cut-off function on scale $b$. Namely, assume that $\phi\in C^{\infty}(\mathbb{T}^2)$, $1\geq \phi(x)\geq 0$, $\phi=1$ on $B_b$, and $\phi=0$ on $B_{2b}^c$. The function $\phi$ can be chosen so that for any multi-index $\alpha\in \mathbb{Z}^2$, $|D^{\alpha}\phi|\leq Cb^{-|\alpha|}$. The parameters $a$, $M$ and $b$ will be chosen below. The local existence of smooth solution $\rho(x,t)$ can be proved by standard method, see e.g. \cite{kiselev2012biomixing} for a closely related argument in the $\mathbb{R}^2$ setting. It is straightforward to check using parabolic comparison principles that if $\rho_0\geq 0$, then $\rho(x,t)\geq 0$ for all $t\geq 0$. Also, we have $\int_{\mathbb{T}^2}\rho(x,t)\, dx=M,$ at least while $\rho(x,t)$ remain
 s smooth.

The first quantity we would like to consider is $\int_{\mathbb{T}^2}\rho(x,t)\phi(2x)\, dx$. We need an estimate showing that the mass cannot leave $B_b$ too quickly.
we note that the constants $C_k$ employed later in this section are not related to the constants $C_k$ in the previous sections.

\begin{lem}\label{2.3}
Suppose that $a$, $b$, $\phi$, $M$ and $\rho_0$ are as described above. Assume that the local solution $\rho(x,t)$ exists and remains regular in the time interval $[0,\tau]$. Then we have
\begin{equation}\label{6th}
\int_{\mathbb{T}^2}\rho(x,t)\phi(2x)\, dx\geq M-C_1M^2b^{-2}t
\end{equation}
for every $t\in [0,\tau]$.
\end{lem}

Naturally, the bound (\ref{6th}) is only interesting if $t$ is sufficiently small.

\begin{proof}
We have
$$
\partial_t \int_{\mathbb{T}^2} \rho(x)\phi(2x)\, dx=\int_{\mathbb{T}^2}\Delta \rho(x)\phi(2x)\, dx-\int_{\mathbb{T}^2} \phi(2x)\nabla\cdot(\rho(x)\nabla(-\Delta)^{-1}(\rho(x)-\bar{\rho}))\, dx.
$$
First, using the periodic boundary conditions and integrating by parts, we find that
\begin{equation}\label{7th}
\left|\int_{\mathbb{T}^2} \Delta\rho(x)\phi(2x)\, dx\right|=4\left|\int_{\mathbb{T}^2}\rho(x)\Delta\phi(2x)\, dx\right|\leq CMb^{-2}.
\end{equation}
Next, let $\psi\in C^{\infty}_0(\mathbb{R}^2)$ be a cutoff function, $\psi(x)=1$ if $|x|\leq 1/2$, $\psi(x)=0$ if $|x|\geq 1$, $0\leq \psi(x)\leq 1$, $|\nabla\psi(x)|\leq C$. Using Lemma \ref{2.2}, we have

\begin{align*}
&\left|\int_{\mathbb{T}^2} \phi(2x)\nabla\cdot(\rho(x)\nabla(-\Delta)^{-1}(\rho(x)-\bar{\rho}))\, dx \right|\\
&=\frac{1}{\pi}\left|\int_{\mathbb{T}^2}\nabla \phi(2x)\rho(x,t)\lim_{\gamma\rightarrow 0+}\int_{\mathbb{R}^2}\frac{(x-y)}{|x-y|^2}(\rho(y,t)-\bar{\rho})e^{-\gamma|y|}\,dy dx\right|\\
&\leq C\left|\int_{\mathbb{T}^2}\nabla \phi(2x)\rho(x)\int_{\mathbb{T}^2}\frac{(x-y)}{|x-y|^2}(\rho(y)-\bar{\rho})\psi(y)\,dydx\right|\\
&+C\left|\int_{\mathbb{T}^2}\nabla \phi(2x)\rho(x)\lim_{\gamma\rightarrow 0+}\int_{\mathbb{R}^2}\frac{(x-y)}{|x-y|^2}(\rho(y)-\bar{\rho})e^{-\gamma |y|^2}(1-\psi(y))\,dydx\right|\\
&=C(I)+C(II).
\end{align*}
We passed to the limit $\gamma\rightarrow 0+$ in the first integral since the integral of the limit converges absolutely. Using symmetrization, we can estimate
\begin{align*}
(I)&\leq\bar{\rho}\left|\int_{\mathbb{R}^2}(\nabla\phi)(2x)\rho(x,t)\int_{\mathbb{R}^2}\frac{(x-y)}{|x-y|^2}\psi(y)\,dydx\right|\\
&+\int_{\mathbb{R}^2}\int_{\mathbb{R}^2} \rho(x)\rho(y)\left|\left[\frac{(x-y)\cdot(\nabla\phi(2x)\psi(y)-\nabla \phi(2y)\psi(x))}{|x-y|^2}\right]\right|\,dxdy\\
&\leq CM^2b^{-1}+ \int_{B_1}\int_{B_1}\rho(x,t)\rho(y,t)(\|\nabla^2\phi\|_{L^{\infty}}+\|\nabla\phi\|_{L^{\infty}}\|\nabla\psi\|_{L^{\infty}})\,dxdy\\
&\leq CM^2 b^{-2}.
\end{align*}
We use the fact that supp$\phi\subset$ supp$\psi\subset B_1$ in the second step.

Next let us estimate $(II)$. Note that in this case the kernel is not singular, since supp$\phi\subset B_{2b}$ while supp$(1-\psi)\subset B_1^c$. However, there is an issue of convergence of $y$ integral over the infinite region. Suppose that $x\in B_{2b}$. Define mean zero function $g$ via $\rho(y)-\bar{\rho}=\Delta g$. In fact, we have $\hat{g}(k)=-\hat{\rho}(k)(2\pi |k|^2)^{-1}$ for $k\neq 0$. By working on the Fourier side, it is easy to show that $\|g\|_{L^1(\mathbb{T}^2)}\leq \|g\|_{L^2(\mathbb{T}^2)}\leq C\|\rho\|_{L^1(\mathbb{T}^2)}$. Now we can estimate
\begin{equation}\label{8th}
\begin{split}
&\lim_{\gamma\rightarrow 0+}\int_{\mathbb{R}^2}\frac{(x-y)}{|x-y|^2}(\rho(y,t)-\bar{\rho})e^{-\gamma|y|^2}(1-\psi(y))\,dxdy\\
&=\lim_{\gamma\rightarrow 0+}\int_{\mathbb{R}^2}\frac{(x-y)}{|x-y|^2}\Delta g(y,t)e^{-\gamma|y|^2}(1-\psi(y))\,dxdy\\
&=\lim_{\gamma\rightarrow 0+}\int_{\mathbb{R}^2}g(y,t)\Bigg( \Delta\left(\frac{(x-y)}{|x-y|^2}\right)e^{-\gamma|y|^2}(1-\psi(y))\\
&+2\nabla \left(\frac{(x-y)}{|x-y|^2}\right)\nabla\left(e^{-\gamma|y|^2}(1-\psi(y))\right)+\frac{(x-y)}{|x-y|^2}\Delta\left(e^{-\gamma|y|^2}(1-\psi(y))\right)\Bigg)\,dxdy.
\end{split}
\end{equation}
Here $g(y,t)$ is extended periodically to the whole $\mathbb{R}^2$. Note that in the first summand in the last integral in (\ref{8th}) we can pass to the limit as $\gamma\rightarrow 0$ since $\Delta \left(\frac{(x-y)}{|x-y|^2}\right)$ decays sufficiently fast. For every $x\in B_b$, we obtain an integral which is bounded by $C\|g\|_{L^1}\sum_{n\in \mathbb{Z}^2,|n|>0}|n|^{-3}\leq CM$. It is straightforward to estimate that the last two summands in (\ref{8th}) are bounded by
\begin{align*}
C\int_{B_{1/2}^c}|g(y,t)|(\gamma|y|^{-1}+\gamma^2|y|)e^{-\gamma|y|^2}(1-\psi(y))\,dy\leq\\
C\|g\|_{L^1(\mathbb{T}^2)}\sum_{n\in \mathbb{Z}^2, |n|>0}(\gamma|n|^{-1}+\gamma^2|n|)e^{-\gamma|n|^2}\leq C\|g\|_{L^1(\mathbb{T}^2)}\gamma^{1/2}\xrightarrow{\gamma\rightarrow 0} 0.
\end{align*}
Combining these estimates, we see that
$$
(II)\leq CM\int_{\mathbb{R}^2}(\nabla \phi)(2x)\rho(x,t)\,dx\leq CM^2b^{-1}.
$$
Therefore, for all times where smooth solution is still defined, and under our assumptions on values of parameters, we have
$$
\left|\partial_t\int_{\mathbb{T}^2}\rho(x,t)\phi(2x)\,dx\right|\leq CM^2b^{-2}. 
$$
This implies (\ref{6th}) and finishes the proof of the lemma.
\end{proof}

Let us now consider the second moment $\int_{\mathbb{T}^2}|x|^2\rho(x,t)\phi(x)\,dx$. Closely related quantities are well-known tools to establish finite time blow up in Keller-Segel equation; see e.g. \cite{perthame2006transport}, \cite{nagai2001blowup}. We have the following lemma.

\begin{lem}\label{2.4}
Suppose $1/4\geq b>0$ and $\phi$ is a cutoff function on scale $b$ as described above. Let $\rho_0\in C^{\infty}(\mathbb{T}^2)$, and assume that the unique local smooth solution $\rho(x,t)$ to (\ref{chemo}) set on $\mathbb{T}^2$ is defined on $[0,T]$. Then for every $t\in[0,T]$ we have
\begin{equation}\label{9th}
\partial_t \int_{\mathbb{T}^2} |x|^2 \rho(x,t) \phi(x)\, dx\leq -\frac{1}{2\pi}\left(\int_{\mathbb{T}^2} \rho(x)\phi(x)\, dx\right)^2+C_2M\|\rho\|_{L^1(\mathbb{T}^2\setminus B_b)}+C_3bM^2+C_4M. 
\end{equation}
\end{lem}
\begin{proof}
In the estimate below, we will use the formula (\ref{infinsum}) with $\gamma$ set to be zero. All the estimates can be done completely rigorously similar to the proof of Lemma \ref{2.3}; we will proceed with the formal computation to reduce repetitive technicalities.

We have

\begin{align*}
\partial_t \int_{\mathbb{T}^2}|x|^2\rho(x,t)\phi(x)\, dx&=\int_{\mathbb{T}^2} |x|^2\Delta \rho(x)\phi(x)\, dx \\ 
&+\int_{\mathbb{T}^2}|x|^2\phi\nabla\cdot(\rho\nabla(-\Delta)^{-1}(\rho-\bar{\rho}))\, dx\\
&=4\int_{\mathbb{T}^2}\phi\rho \, dx+\int_{\mathbb{T}^2}|x|^2\Delta\phi\rho \, dx+4\int_{\mathbb{T}^2} (x\cdot \nabla \phi)\rho \, dx\\
&-\frac{1}{\pi}\int_{\mathbb{T}^2} \phi(x)\rho(x)\int_{\mathbb{R}^2}\frac{x(x-y)}{|x-y|^2}(\rho(y)-\bar{\rho})\,dydx\\
&-\frac{1}{2\pi}\int_{\mathbb{T}^2}|x|^2\rho(x)\int_{\mathbb{R}^2}\frac{\nabla\phi(x)\cdot(x-y)}{|x-y|^2}(\rho(y)-\bar{\rho})\,dxdy.\\
&\equiv(i)+(ii)+(iii).
\end{align*}
Here $(i)$ denotes the first three terms. By our choice of $\phi$, $(i)$ does not exceed $C_4M$ for some constant $C_4$. Next, let us write
\begin{equation}\label{10th}
\begin{split}
(ii)=&-\frac{1}{\pi}\int_{\mathbb{T}^2} \phi(x)\rho(x,t)\int_{\mathbb{R}^2}\frac{x(x-y)}{|x-y|^2}(\rho(y,t)-\bar{\rho})\psi(y)\,dydx\\
&-\frac{1}{\pi}\int_{\mathbb{T}^2} \phi(x)\rho(x,t)\int_{\mathbb{R}^2}\frac{x(x-y)}{|x-y|^2}(\rho(y,t)-\bar{\rho})(1-\psi(y))\,dydx,
\end{split}
\end{equation}
where $\psi$ is a cutoff function as in Lemma \ref{2.3}. The absolute value of the integral
$$
\int_{\mathbb{T}^2} \phi(x)\rho(x,t)\int_{\mathbb{R}^2}\frac{x(x-y)}{|x-y|^2}(\rho(y,t)-\bar{\rho})(1-\psi(y))\,dydx
$$
can be controlled similarly to the estimates applied in bounding the term $(II)$ in the proof of Lemma \ref{2.3}, leading to an upper bound by $CM^2b$. Next, we can estimate
$$
\left|\bar{\rho}\int_{\mathbb{T}^2}\phi(x)\rho(x,t)\int_{\mathbb{R}^2}\frac{x(x-y)}{|x-y|^2}\psi(y)\,dxdy\right|\leq CM^2b
$$
as well. Split the remaining part of the first integral in (\ref{10th}) into two parts:
\begin{align*}
&-\frac{1}{\pi}\int_{\mathbb{T}^2}\phi(x)\rho(x,t)\int_{\mathbb{R}^2}\frac{x(x-y)}{|x-y|^2}\rho(y,t)\phi(y)\,dxdy\\
&-\frac{1}{\pi}\int_{\mathbb{T}^2}\phi(x)\rho(x,t)\int_{\mathbb{R}^2}\frac{x(x-y)}{|x-y|^2}\rho(y,t)(1-\phi(y))\psi(y)\,dxdy
\end{align*}
Using symmetrization, we obtain

\begin{align*}
&-\frac{1}{\pi}\int_{\mathbb{T}^2} \phi(x)\rho(x)\int_{\mathbb{R}^2}\frac{x(x-y)}{|x-y|^2}\rho(y)\phi(y)\,dydx\\
&=-\frac{1}{2\pi}\int_{\mathbb{T}^2}\int_{\mathbb{T}^2}\phi(x)\rho(x)\phi(y)\rho(y)\,dxdy=-\frac{1}{2\pi}\left(\int_{\mathbb{T}^2} \rho(x)\phi(x)\,dx\right)^2.
\end{align*}
On the other hand, 
\begin{align*}
&\int_{\mathbb{T}^2}\phi(x)\rho(x,t)\int_{\mathbb{T}^2}\frac{x(x-y)}{|x-y|^2}\rho(y,t)(1-\phi(y))\psi(y)\,dxdy\\
&=\frac{1}{2}\int_{\mathbb{R}^2}\int_{\mathbb{R}^2} \frac{\rho(x,t)\rho(y,t)}{|x-y|^2}(x-y)\cdot[x\phi(x)(1-\phi(y))\psi(y)-y\phi(y)(1-\phi(x))\psi(x)]\,dxdy.\\
\end{align*}
Let us define $F(x,y)=x\phi(x)(1-\phi(y))\psi(y)-y\phi(y)(1-\phi(x))\psi(x)$. Observe that $F(x,y)=0$ on $B_b\times B_b$, $F(x,x)=0$ and $|\nabla F(x,y)|\leq C$ for all $x,y$. This means
\begin{align*}
|F(x,y)|=|F(x,y)-F(x,x)|&\leq \|\nabla F\|_{L^{\infty}}|x-y|\chi_{B_1\times B_1\setminus B_b\times B_b}(x,y)\\
&\leq C|x-y|\chi_{B_1\times B_1\setminus B_b\times B_b}(x,y),
\end{align*}
where $\chi_{S}(x,y)$ denotes the characteristic function of a set $S\subset \mathbb{R}^2\times \mathbb{R}^2$. Therefore,
\begin{align*}
&\int_{\mathbb{T}^2}\phi(x)\rho(x,t)\int_{\mathbb{R}^2}\frac{x(x-y)}{|x-y|^2}\rho(y)(1-\phi(y))\psi(y)\,dxdy\\
&\leq C\int\int_{B_1\times B_1\setminus B_b\times B_b}\rho(x,t)\rho(y,t)\,dxdy\\
&\leq CM\|\rho\|_{L^1(D\setminus B_b(0))}.
\end{align*}


To summarize, $(ii)$ can be bounded above by
$$
-\frac{1}{2\pi}\left(\int_{\mathbb{T}^2}\rho(x,t)\phi(x)\,dx\right)^2+CM\|\rho(\cdot, t)\|_{L^1(\mathbb{T}^2\setminus B_b)}+CbM^2.
$$

Finally, let us estimate $(iii)$. Similarly to the previous part, we have
$$
\int_{\mathbb{T}^2}|x|^2\rho(x,t)\int_{\mathbb{R}^2}\frac{\nabla\phi(x)\cdot (x-y)}{|x-y|^2}(\rho(y,t)-\bar{\rho})(1-\psi(y))\,dxdy\leq CbM^2.
$$
Also,
$$
\bar{\rho}\int_{\mathbb{T}^2}|x|^2\rho(x,t)\int_{\mathbb{R}^2}\frac{\nabla\phi(x)\cdot (x-y)}{|x-y|^2}(1-\psi(y))\,dxdy\leq CbM^2
$$
as well. The remaining part of $(iii)$ we can estimate by using symmetrization:
\begin{align*}
&\int_{\mathbb{R}^2}|x|^2 \rho(x,t)\int_{\mathbb{R}^2}\frac{\nabla\phi(x)\cdot (x-y)}{|x-y|^2}\rho(y,t)\psi(y)\,dxdy\\
&=\frac{1}{2}\int_{\mathbb{R}^2}\int_{\mathbb{R}^2}\rho(x,t)\rho(y,t)\left(\frac{(x-y)\cdot (\nabla \phi(x)\psi(y)|x|^2-\nabla \phi(y)\psi(x)|y|^2)}{|x-y|^2}\right)\,dxdy.
\end{align*}
Observe that
\begin{align*}
||x|^2\nabla\phi(x)\psi(y)-|y|^2\nabla\phi(y)\psi(x)|\leq C\chi_{B_{2b}\times B_{2b}\setminus B_b\times B_b}|x-y|.
\end{align*}
Therefore 
$$
(iii)\leq CbM^2+CM\|\rho\|_{L^1(\mathbb{T}^2\setminus B_b)}.
$$
Combining the estimate of $(i)$, $(ii)$ and $(iii)$ yields (\ref{9th}), proving the lemma.
\end{proof}

We are now ready to complete the proof of Theorem \ref{blowupt}.
\begin{proof}[Proof of Theorem \ref{blowupt}]
Let us recall that we assume $1/4\geq b \geq 2a$, and the initial data $\rho_0$ is supported inside $B_a$. Assume that the unique solution $\rho(x,t)$ of (\ref{chemo}) set on $\mathbb{T}^2$ remains smooth for all $t$. Then by Lemma \ref{2.3}, and conservation of mass, for all $t\geq 0$ we have
$$
\|\rho(\cdot, t)\|_{L^1(\mathbb{T}^2\setminus B_b)}\leq M-\int_{\mathbb{T}^2}\rho(x,t)\phi(2x)\,dx\leq C_1M^2 b^{-2}t. 
$$
Also,
$$
\int_{\mathbb{T}^2}\rho(x,t)\phi(x)dx\geq \int_{\mathbb{T}^2}\rho(x,t)\phi(2x)\,dx\geq M-C_1M^2b^{-2}t. 
$$
Therefore, by Lemma \ref{2.4}, we have that
\begin{align*}
\partial_t \int_{\mathbb{T}^2} |x|^2 \rho(x,t) \phi(x)\,dx\leq -\frac{1}{2\pi}(M-C_1M^2b^{-2}t)^2+C_2M^3b^{-2}t\\
+C_3M^2b+C_4M
\end{align*}
for all $0\leq t\leq \frac{b^2}{C_1M}$. We will now make the choice of all our parameters.\\
\noindent
1. Choose $b$ so that $C_3b\leq 0.001$.\\
\noindent
2. Choose $M$ so that $M\geq 1000 C_4$.\\
\noindent
3. Choose $a$ so that the following three inequalities hold:
$$
a\leq b/2, \,\,\, a \leq \frac{b}{10\sqrt{2C_1}}, \mbox{ and } a\leq \frac{b}{100\sqrt{C_2}}.
$$
\noindent
4. Choose the time $\tau=\frac{100a^2}{M}$.

With such choice of parameters, it is straightforward to check that
\begin{equation}\label{11th}
\partial_t \int_{\mathbb{T}^2} |x|^2 \rho(x,t) \phi(x)\,dx\leq -\frac{M^2}{50}
\end{equation}
for every $t\in[0,\tau]$. But by assumption, supp$\rho_0\subset B_a$, and so
\begin{equation}\label{12th}
\int_{\mathbb{T}^2} |x|^2 \rho_0(x) \phi(x)\,dx\leq a^2 M.
\end{equation}
Together, (\ref{11th}), (\ref{12th}) and our choice of $\tau$ imply that $\int_{\mathbb{T}^2}|x|^2\rho(x,\tau)\phi(x)\,dx$ must be negative. This is a contradiction with the assumption that $\rho(x,t)$ stays smooth throughout $[0,\tau]$.
\end{proof}

In fact, it is not hard to verify that finite time blow up persists if we add an advection term to the Keller-Segel equation, but change the order of the selection of $u$ and $\rho_0$ compared with the results
on suppression of chemotactic explosion. Namely, the following theorem holds. 
\begin{thm}\label{buu}
Consider the Keller-Segel equation \eqref{chemo1} set on $\Tm^2.$ Suppose that $u(x,t) \in C^\infty(\Tm^2 \times [0,\infty))$ is incompressible. Then there exists $\rho_0 \in C^\infty(\Tm^2),$ $\rho_0 \geq 0,$ such that
the corresponding solution $\rho(x,t)$ of \eqref{chemo1} blows up in finite time.
\end{thm}
The proof of Theorem~\ref{buu} closely follows that of Theorem~\ref{blowupt}. The advection term can be easily estimated and the proof requires only a few adjustments of constants. We omit the details.

\section{Appendix II: Some inequalities}

Here we first prove Proposition \ref{7.5}, and then sketch the proofs of the multiplicative inequalities (\ref{14th}) and (\ref{16th}).

Let $\rho(x,t)$ be a solution of (\ref{chemo1})
$$
\partial_t\rho+(u\cdot\nabla) \rho -\Delta \rho + \nabla\cdot(\rho \nabla (-\Delta)^{-1}(\rho-\bar{\rho}))=0,\quad \rho(x,0)=\rho_0(x),
$$
with some smooth incompressible vector field $u(x,t)$.
\begin{prop}\label{9.1}
Let $\rho_0\in C^{\infty}(\mathbb{T}^2)$. Suppose that $\|\rho(\cdot, t)-\bar{\rho}\|_{L^2}\leq 2B$ for all $t\in [0,T]$ and some $B\geq 1$. Then we also have $\|\rho(\cdot, t)-\bar{\rho}\|_{L^{\infty}}\leq C_4B\max(B,\bar{\rho}^{1/2})$ for some universal constant $C_4$ and all $t\in [0,T]$.
\end{prop}
\begin{proof}
Observe first that by a direct computation, for every integer $p\geq 1$ we have
\begin{equation}\label{64th}
\begin{split}
\partial_t \int_{\mathbb{T}^2}(\rho-\bar{\rho})^{2p}\,dx=-(4-\frac{2}{p})\int_{\mathbb{T}^2}|\nabla ((\rho-\bar{\rho})^p)|^2\,dx\\
+(2p-1)\int_{\mathbb{T}^2}(\rho-\bar{\rho})^{2p+1}\,dx+2p\bar{\rho}\int_{\mathbb{T}^2}(\rho-\bar{\rho})^{2p}\,dx.
\end{split}
\end{equation}
To obtain (\ref{64th}), we need to multiply (\ref{chemo1}) by $2p(\rho-\bar{\rho})^{2p-1}$, integrate, and then simplify the obtained terms using integration by parts and the fact that $u$ is divergence free.

Let us estimate $\|\rho-\bar{\rho}\|_{L^{2^n}}$ inductively. By assumption, we have $\|\rho(\cdot, t)-\bar{\rho}\|_{L^2}\leq 2B$ for all $t\in [0,T]$. Assume that for some $n\geq 1$, we have
$$
\|\rho-\bar{\rho}\|_{L^{2^n}}\leq \Upsilon_n \quad\mbox{for all }t\in [0,T].
$$
Let us derive an estimate for an upper bound $\Upsilon_{n+1}$ on $\|\rho-\bar{\rho}\|_{L^{2^{n+1}}}$ on $[0,T]$. For that purpose, let us set $p=2^n$ in (\ref{64th}) and let us define $f(x,t)=(\rho-\bar{\rho})^p\equiv (\rho-\bar{\rho})^{2^n}$. Then (\ref{64th}) implies
\begin{equation}\label{65th}
\partial_t\int_{\mathbb{T}^2} |f|^2\,dx\leq -2 \int_{\mathbb{T}^2}|\nabla f|^2\,dx+ 2^{n+1}\int_{\mathbb{T}^2}|f|^{2+2^{-n}}\,dx+2^{n+1} \bar{\rho}\int_{\mathbb{T}^2}|f|^2.
\end{equation}
Also, in terms of $f$, our induction assumption is that $\int_{\mathbb{T}^2}|f|\,dx\leq \Upsilon_n^{2^{n}}$.

We will now need the following Gagliardo-Nirenberg inequality.

\begin{lem}\label{9.2}
Suppose $v\in C^{\infty}(\mathbb{T}^d)$, $d\geq 2$, and the set where $v$ vanishes is nonempty. Assume that $q,r>0$, $\infty>q>r$, and $\frac{1}{d}-\frac{1}{2}+\frac{1}{r}>0$. Then
\begin{equation}\label{66th}
\|v\|_{L^q}\leq C(d,q)\|\nabla v\|_{L^2}^a\|v\|_{L^r}^{1-a},\quad a=\frac{\frac{1}{r}-\frac{1}{q}}{\frac{1}{d}-\frac{1}{2}+\frac{1}{r}}.
\end{equation}
The constant $C(d,q)$ for a fixed $d$ is bounded uniformly when $q$ varies in any compact set in $(0,\infty)$.
\end{lem}
\begin{proof}
This inequality is well known in the case $v\in C^{\infty}_0(\mathbb{R}^d)$, see e.g. \cite{maz2013sobolev}. A simple proof is contained in \cite{kiselev2012biomixing}. Going through the proof in \cite{kiselev2012biomixing}, it is not difficult to verify that the result still holds in the periodic case under the assumption that $v$ vanishes somewhere in $\mathbb{T}^d$ (which rules out increasing mean value without increasing variance). One can similarly trace the claim regarding the constant $C(d,q)$. We refer to \cite{kiselev2012biomixing} for details.
\end{proof}

Applying Lemma \ref{9.2} with $d=2$, $r=2$, and $q=2+2^{-n}$ yields
\begin{equation}\label{67th}
\|f\|_{L^{2+2^{-n}}}^{2+2^{-n}}\leq C \|\nabla f\|_{L^2}^{2^{-n}}\|f\|_{L^2}^2\leq \frac{1}{2^{n+1}}\|\nabla f\|_{L^2}^2+C\|f\|_{L^2}^{\frac{2}{1-2^{-n-1}}},
\end{equation}
where we used Young's inequality in the last step. Moreover, we also have
\begin{equation}\label{68th}
\|f\|_{L^2}\leq C\|\nabla f\|_{L^2}^{1/2}\|f\|_{L^1}^{1/2}.
\end{equation}
Applying (\ref{68th}) and (\ref{67th}) to (\ref{65th}), we obtain
\begin{equation}\label{69th}
\partial_t \|f\|_{L^2}^2\leq -C_1\|f\|_{L^2}^4\|f\|_{L^1}^{-2}+C_2 2^{n+1}\|f\|_{L^2}^{\frac{2}{1-2^{-n-1}}}+2^{n+1}\bar{\rho}\|f\|_{L^2}^2,
\end{equation}
where $C_{1,2}$ are some fixed universal constants (not connected to $C_1$ and $C_2$ used earlier in the paper). Clearly, given the upper bound on $\|f\|_{L^1}$, the right hand side of (\ref{69th}) turns negative if $\|f\|_{L^{2}}$ becomes sufficiently large. Thus $\|f\|_{L^2}$ cannot cross this threshold. Assuming without loss of generality that $\Upsilon_n\geq 1$ for all $n$, a direct computation shows that if $\|\rho-\bar{\rho}\|_{L^{2^{n+1}}}$ reaches the value $\Upsilon_{n+1}$ which satisfies the following recursive equality, then the right hand side of (\ref{69th}) is negative:
\[ \log \Upsilon_{n+1} = {\rm max}(\Gamma_n, \Theta_n) \]
where
\begin{equation}\label{70th}
\Gamma_n =\frac{2^{n+1}-1}{2^{n+1}-2}\log \Upsilon_n+\frac{1}{2^{n+1}}((n+1)\log 2+\log C)
\end{equation}
\begin{equation}\label{71th}
\Theta_n =\log \Upsilon_n + \frac{1}{2^{n+1}}((n+1)\log 2+\log C +\max (\log \bar{\rho},0)).
\end{equation}
Here $C\geq 1$ is some universal constant. Denote $q_j=\frac{2^{j+1}-1}{2^{j+1}-2}$ and observe that due to telescoping,
$$
\prod_{j=1}^n q_j = \frac{2^{n+1}-1}{2^n}\xrightarrow{n\rightarrow \infty}2.
$$
An elementary inductive computation shows that if $B\gtrsim\bar{\rho}^{1/2}$, then the first recursive relation (\ref{70th}) determines the size of $\Upsilon_{n+1},$
yielding the estimate $\Upsilon_{n+1}\leq CB^2$. If $B\lesssim \bar{\rho}^{1/2}$, then the second relation (\ref{71th}) dominates, yielding the estimate $\Upsilon_{n+1}\leq CB\bar{\rho}^{1/2}$. Since
$$
\|\rho-\bar{\rho}\|_{L^{\infty}}=\lim_{n\rightarrow \infty}\|\rho-\bar{\rho}\|_{L^{2^n}},
$$
we obtain that
$$
\|\rho-\bar{\rho}\|_{L^{\infty}}\leq CB\max(B,\bar{\rho}),
$$
proving the proposition.
\end{proof}
\begin{prop}
Suppose that $f\in C^{\infty}(\mathbb{T}^d)$ and is mean zero. Then
$$
\|D^mf\|_{L^p}\leq C\|f\|_{L^2}^{1-a}\|f\|_{\dot{H}^n}^a, \quad a=\frac{m-\frac{d}{p}+\frac{d}{2}}{n},
$$
where $D$ stands for any partial derivative, $2\leq p\leq \infty$, and we assume $n>m+d/2$.
\end{prop}

Note that the last assumption is not necessary. However it makes the proof simpler, and this is the only case we need in this paper. Indeed, in the estimates of Section 3 we have $s+1>l+d/2$ unless $l=s$ (recall $d\leq 3$). But when $l=s$, we are only estimating $\|D^s\rho\|_{L^2}$, which is straightforward.
\begin{proof}
Consider $p=2$. Then
\begin{equation}\label{72th}
\|D^mf\|_{L^2}\leq \|f\|_{L^2}^{1-\frac{m}{n}}\|f\|_{\dot{H}^n}^{\frac{m}{n}}
\end{equation}
by H\"older inequality on Fourier side.

Next consider $p=\infty$. Then
$$
\|D^mf\|_{L^{\infty}}\leq C\sum_{0<|k|<\Lambda}|k|^m|\hat{f}(k)|+C\sum_{|k|\geq \Lambda}|k|^m|\hat{f}(k)|\equiv (I)+(II).
$$
Now
$$
(I)\leq C\Lambda^{m+\frac{d}{2}}\left(\sum_{0<|k|<\Lambda}|\hat{f}(k)|^2\right)^{1/2}
$$
by Cauchy-Schwartz. On the other hand,
$$
(II)\leq C\left(\sum_{|k|\geq \Lambda}|k|^{2n}|\hat{f}(k)|^2\right)^{1/2}\left(\sum_{|k|\geq \Lambda}|k|^{2(m-n)}\right)^{1/2}\leq C\|f\|_{\dot{H}^n}\Lambda^{(m-n)+\frac{d}{2}},
$$
provided that $n>m+\frac{d}{2}$. Choose $\Lambda$ so that
$$
\|f\|_{L^2}\Lambda^{m+\frac{d}{2}}=\|f\|_{\dot{H}^n}\Lambda^{m-n+\frac{d}{2}}.
$$
Such choice leads to the bound
\begin{equation}\label{73th}
\|D^mf\|_{L^{\infty}}\leq C \|f\|_{L^2}^{\frac{n-m+d/2}{n}}\|f\|_{\dot{H}^n}^{\frac{m+d/2}{n}}.
\end{equation}

The general case $2<p<\infty$ follows immediately from (\ref{72th}) and (\ref{73th}).
\end{proof}
\begin{prop}
Suppose that $f\in C^{\infty}(\mathbb{T}^d)$, and $m>0$. Then
$$
\|f\|_{\dot{H}^s}\leq C\|f\|_{\dot{H}^{s+1}}^{\frac{2s+d}{2s+2+d}}\|f\|_{L^1}^{\frac{2}{2s+2+d}}.
$$
\end{prop}
\begin{proof}
The proof of this proposition can be done similarly to the previous one. One needs to use that $\|\hat{f}\|_{L^{\infty}}\leq \|f\|_{L^1}$. We leave details to the interested reader.
\end{proof}


\end{document}